\documentclass[11pt]{article}
\usepackage{graphicx}
\usepackage{amsmath,amsthm,amssymb,enumerate}
\usepackage{euscript,mathrsfs}
\usepackage[left=3cm,right=3cm,top=3.5cm,bottom=3.5cm]{geometry}
\usepackage{color}
\catcode`\@=11 \@addtoreset{equation}{section}

\catcode`\@=12

\allowdisplaybreaks

\newtheorem{Theorem}{Theorem}[section]
\newtheorem{Proposition}[Theorem]{Proposition}
\newtheorem{Lemma}[Theorem]{Lemma}
\newtheorem{Corollary}[Theorem]{Corollary}

\theoremstyle{definition}
\newtheorem{Definition}[Theorem]{Definition}

\newtheorem{Remark}[Theorem]{Remark}

\newcommand{\bTheorem}[1]{
\begin{Theorem} \label{T#1} }
\newcommand{\eT}{\end{Theorem}}

\newcommand{\bProposition}[1]{
\begin{Proposition} \label{P#1}}
\newcommand{\eP}{\end{Proposition}}

\newcommand{\bLemma}[1]{
\begin{Lemma} \label{L#1} }
\newcommand{\eL}{\end{Lemma}}

\newcommand{\bCorollary}[1]{
\begin{Corollary} \label{C#1} }
\newcommand{\eC}{\end{Corollary}}

\newcommand{\bRemark}[1]{
\begin{Remark} \label{R#1} }
\newcommand{\eR}{\end{Remark}}

\newcommand{\bDefinition}[1]{
\begin{Definition} \label{D#1} }
\newcommand{\eD}{\end{Definition}}

\newcommand{\dif}{\mathrm{d}}

\newcommand{\mf}{\mathscr{F}}
\newcommand{\mr}{\mathbb{R}}
\newcommand{\prst}{\mathbb{P}}
\newcommand{\p}{\mathbb{P}}
\newcommand{\stred}{\mathbb{E}}
\newcommand{\ind}{\mathbf{1}}
\newcommand{\mn}{\mathbb{N}}

\newcommand{\mt}{\mathbb{T}^N}

\newcommand{\bu}{\mathbf u}

\newcommand{\bq}{\mathbf q}

\newcommand{\tor}{\mathbb{T}^N}

\newcommand{\StoB}{\left(\Omega, \mathfrak{F},(\mathfrak{F}_t )_{t \geq 0},  \mathbb{P}\right)}

\newcommand{\bfw}{\mathbf{w}}
\newcommand{\bfu}{\mathbf{u}}
\newcommand{\bfv}{\mathbf{v}}
\newcommand{\bfq}{\mathbf{q}}

\newcommand{\bfr}{\mathbf{r}}

\newcommand{\bfa}{\mathbf{a}}

\newcommand{\bfY}{\mathbf{Y}}
\newcommand{\bfZ}{\mathbf{Z}}
\newcommand{\bfF}{\mathbf{F}}
\newcommand{\bfC}{\mathbf{C}}

\newcommand{\bfpsi}{\boldsymbol{\psi}}

\newcommand{\ds}{\,\mathrm{d}\sigma}

\newcommand{\bFormula}[1]{
\begin{equation} \label{#1}}
\newcommand{\eF}{\end{equation}}

\newcommand{\Ov}[1]{\overline{#1}}

\newcommand{\vr}{\varrho}

\newcommand{\vu}{\vc{u}}
\newcommand{\vc}[1]{{\bf #1}}

\newcommand{\Div}{{\rm div}_x}
\newcommand{\Grad}{\nabla_x}

\newcommand{\tn}[1]{\mathbb{#1}}
\newcommand{\dx}{\,{\rm d} {x}}

\newcommand{\dt}{\,{\rm d} t }

\newcommand{\vv}{\vc{v}}

\newcommand{\D}{{\rm d}}

\newcommand{\R}{\mathbb{R}}

\newcommand{\E}{\mathbb{E}}

\definecolor{Cgrey}{rgb}{0.85,0.85,0.85}
\definecolor{Cblue}{rgb}{0.50,0.85,0.85}
\definecolor{Cred}{rgb}{1,0,0}
\definecolor{fancy}{rgb}{0.10,0.85,0.10}

\newcommand\Cbox[2]{%
    \newbox\contentbox%
    \newbox\bkgdbox%
    \setbox\contentbox\hbox to \hsize{%
        \vtop{
            \kern\columnsep
            \hbox to \hsize{%
                \kern\columnsep%
                \advance\hsize by -2\columnsep%
                \setlength{\textwidth}{\hsize}%
                \vbox{
                    \parskip=\baselineskip
                    \parindent=0bp
                    #2
                }%
                \kern\columnsep%
            }%
            \kern\columnsep%
        }%
    }%
    \setbox\bkgdbox\vbox{
        \color{#1}
        \hrule width  \wd\contentbox %
               height \ht\contentbox %
               depth  \dp\contentbox
        \color{black}
    }%
    \wd\bkgdbox=0bp%
    \vbox{\hbox to \hsize{\box\bkgdbox\box\contentbox}}%
    \vskip\baselineskip%
}

\date{}

\begin{document}


\title{Local strong solutions to the stochastic compressible Navier--Stokes system}

\author{Dominic Breit \and Eduard Feireisl
\thanks{The research of E.F.~leading to these results has received funding from the
European Research Council under the European Union's Seventh
Framework Programme (FP7/2007-2013)/ ERC Grant Agreement
320078. The Institute of Mathematics of the Academy of Sciences of
the Czech Republic is supported by RVO:67985840.}
\and Martina Hofmanov\' a}

\date{\today}

\maketitle

\centerline{Department of Mathematics, Heriot-Watt University}

\centerline{Riccarton Edinburgh EH14 4AS, UK}

\bigskip

\centerline{Institute of Mathematics of the Academy of Sciences of the Czech Republic}

\centerline{\v Zitn\' a 25, CZ-115 67 Praha 1, Czech Republic}

\bigskip

\centerline{Technical University Berlin, Institute of Mathematics}

\centerline{Stra\ss e des 17. Juni 136, 10623 Berlin, Germany}

\bigskip


\begin{abstract}

We study the Navier--Stokes system {describing} the motion of a compressible viscous fluid driven by a nonlinear multiplicative stochastic force. We establish local in time existence (up to a positive stopping time) of a unique solution, which is strong in both PDE and probabilistic sense. Our approach relies on rewriting the problem as a symmetric hyperbolic system augmented by partial diffusion, which is solved via a suitable approximation procedure using the stochastic compactness method and the Yamada--Watanabe type argument based on the  Gy\"{o}ngy--Krylov characterization of convergence in probability. This leads to the existence of a strong (in the PDE sense) pathwise solution. Finally, we use various stopping time arguments to establish the local existence of a unique strong solution to the original problem.

%

\end{abstract}



\section{Introduction}
\label{P}

Stochastic perturbations in the equations of motions are commonly used to model small perturbations (numerical,
empirical, and physical uncertainties) or thermodynamic fluctuations present in
fluid flows. Moreover, it is used for a better understanding of turbulence.
As a consequence stochastic partial differential equations (SPDEs) such as the stochastic Navier--Stokes equations
are gaining more and more interest in fluid mechanical research. First result can be traced back to the pioneering work by Bensoussan end Teman \cite{BeTe} in 1973.
Today there exists an abundant amount of literature concerning the dynamics of incompressible fluids driven by stochastic forcing. We refer to the lecture notes by Flandoli \cite{Fl}, the monograph of Kuksin and
Shyrikian \cite{KukShi} as well as the references cited therein for a recent overview. Definitely much less is known if compressibility of the fluid is taken into account. Fundamental questions of well--posedness and even mere existence of solutions to problems dealing with stochastic perturbations of compressible fluids are, to the best of our knowledge, largely open, with only a few rigorous results available.\\
First existence results were based on a suitable transformation formula that allows to reduce the problem to a random system of PDEs: The stochastic integral does no longer appear and deterministic methods are applicable, see \cite{MR1760377} for the 1D case, \cite{MR1807944} for a rather special periodic 2D. The latter one is based on
the existence theory developed by Va{\u\i}gant and Kazhikhov in \cite{MR1375428}. Finally, the work by Feireisl, Maslowski, Novotn\'y \cite{MR2997374} deals with the 3D case.
 The first ``truly'' stochastic existence result for the compressible Navier--Stokes system perturbed by a general nonlinear multiplicative noise was obtained by Breit, Hofmanov\'a \cite{BrHo}. The existence of the so-called finite energy weak martingale solutions in three space dimensions with periodic boundary conditions was established. Extension of this result to the zero Dirichlet boundary conditions then appeared in \cite{2015arXiv150400951S,MR3385137}. For completeness, let us also mention \cite{BrFeHo2015B} where a singular limit result was proved.

The next step towards a better understanding of stochastic compressible fluids is the so-called relative energy inequality derived in \cite{BrFeHo2015A}.
Among other possible applications, it allows to compare a weak solution to the compressible system with arbitrary (smooth) processes, in particular
with a strong solution of the same problem. This gives rise to the weak--strong uniqueness principle: A weak (in the PDE sense) solution satisfying the energy inequality necessarily coincides with a strong solution emanating from the same initial data, as long as the latter one exists.
In the light of this result,
a natural question to ask is whether or not a strong solution exists at least locally in time. Results concerning the existence of strong solutions in three dimensions, however, do not exists at all.
In the present paper, we fill this gap by showing existence of local-in-time strong solutions (up to a positive stopping time) of the stochastic
compressible Navier--Stokes system
enjoying the regularity properties required by the weak--strong uniqueness principle established in \cite{BrFeHo2015A}.

We consider a stochastic variant of the \emph{compressible barotropic
Navier-Stokes system} describing the time evolution of the mass density $\vr$ and the bulk velocity $\vu$ of a fluid driven by
a nonlinear multiplicative noise. The system of equations reads
\begin{equation} \label{P1}
\D \vr + \Div (\vr \vu) \ \dt = 0
\end{equation}
\begin{equation} \label{P2}
\D (\vr \vu) + \left[ \Div (\vr \vu \otimes \vu) + a \Grad \vr^\gamma \right]  \dt  = \Div \mathbb{S} (\Grad \vu) \ \dt + \mathbb{G} (\vr, \vr \vu) \D W,
\end{equation}
where $\mathbb{S}(\Grad \vu)$ is the standard Newtonian viscous stress tensor,
\begin{equation} \label{P3}
\mathbb{S}(\Grad \vu) = \mu \left( \Grad \vu + \Grad^t \vu - \frac{2}{3} \Div \vu \mathbb{I} \right) + \lambda \Div \vu \mathbb{I}, \qquad \mu > 0, \ \lambda \geq 0.
\end{equation}
The driving process $W$ is a cylindrical Wiener process defined on some probability space $(\Omega,\mathfrak{F},\p)$ and the coefficient $\mathbb{G}$ is generally nonlinear and satisfies suitable growth assumptions. The precise assumptions will be specified in Section \ref{E}.
We focus on the periodic boundary conditions, for which the underlying spatial domain $\mathcal O \subset \R^N$ may be identified with the flat torus
\[
\mathcal O = \mt = \left( (-\pi, \pi)|_{\{ - \pi, \pi \}} \right)^N ,\ N=1,2,3.
\]
The initial conditions are random variables
\begin{equation} \label{P4}
\vr(0,\cdot) = \vr_0, \ \vu(0, \cdot) = \vu_0,
\end{equation}
with sufficient space regularity specified later.

We study the system \eqref{P1}--\eqref{P4} in the framework of solutions that are strong in both PDE and probabilistic sense.
More precisely, such solutions possess sufficient space regularity for \eqref{P1}--\eqref{P4} to be satisfied pointwise (not only in the sense of distributions) and they are defined on a given probability space.
We introduce the notion of \emph{local strong pathwise solutions} which only exists up to a suitable stopping time, see Definition \ref{def:strsol}.
Next, we consider \emph{maximal strong pathwise solutions} which live on a maximal (random) time interval determined by the hypothetical
blow-up of the $W^{2,\infty}$-norm of the velocity $\vu$, see Definition \ref{def:maxsol}.
Our main result, Theorem \ref{thm:main}, then states the existence of a unique maximal strong pathwise solution to problem \eqref{P1}--\eqref{P4}.

The deterministic approach to the local existence problem
for the compressible Navier-Stokes system is usually based on energy estimates. These are derived first for the unknown functions $\vr$, $\vu$ and then,
repeatedly, for their time derivatives up to a sufficient order to guarantee the required smoothness, see the nowadays probably
optimal result by Cho, Choe and Kim \cite{ChoChoeKim}. However, for obvious reasons related to the irregularity of sample paths of the Brownian motion, this technique is not suitable in the stochastic setting. Instead, the required space regularity must be achieved by differentiating the equations only with respect
to the space variables - a typical approach applicable to purely hyperbolic systems. The related references include works on the incompressible stochastic Navier--Stokes system \cite{BeFr,BrzP}, the incompressible stochastic Euler equations \cite{GHVic}, and also quasilinear hyperbolic systems \cite{JUKim}.

Similarly to Kim \cite{JUKim} (see also \cite{GHVic}), we use suitable cut-off operators to render all non-linearities in the equations
globally Lipschitz. The resulting (stochastic) system may admit \emph{global-in-time solutions}. Still, the approach proposed in
\cite{JUKim} and later revisited in \cite{GHVic} cannot be applied in a direct fashion for the following reasons:

\begin{itemize}
\item The energy method is applicable to \emph{symmetric} hyperbolic systems and their viscous perturbations.
\item In order to symmetrize (\ref{P1}), (\ref{P2}), the density must be strictly positive - the system must be out of vacuum.
\item For the density to remain positive at least on a short time interval, the maximum principle must be applied to the transport
equation (\ref{P1}). Accordingly, equation (\ref{P1}) must be solved exactly and not by means of a finite-dimensional approximation.
\item To avoid technical problems with non-local operators in the transport equation, the cut-off must be applied only to the velocity field.
\end{itemize}

In view of these difficulties and anticipating strict positivity of the density, we transform
the problem to a symmetric hyperbolic system perturbed by partial viscosity and the stochastic driving term, see Subsection \ref{subsec:symsyst}. Then  cut-off operators in the spirit of
\cite{JUKim} are applied to the  velocity field and this system is then studied in detail in Section \ref{A}. We use this technique to cut the nonlinear parts as well as to guarantee the nondegeneracy of the density, which leads
to global in time strong martingale solutions to this approximate system. The main ideas of the proof are as follows. First, we adapt a hybrid method similar to the one proposed in \cite{BrHo}: The equation of continuity is solved directly, while
the momentum equation is approximated by a finite dimensional Galerkin scheme. On this level, we are able to gain higher order uniform energy estimates by differentiating
in space.
Then, using the stochastic compactness method, we  prove the existence of a strong martingale solution.
In Subsection \ref{subsec:uniq}   we establish pathwise uniqueness and then the method of Gy\"ongy--Krylov \cite{krylov} is applied to
recover the convergence of the approximate solutions on the original probability
space, see Subsection \ref{subsec:pathwiseexist}. The existence of a unique strong pathwise solution therefore follows.

Finally, in Section \ref{subsec:strong} we employ the results of the previous sections to prove our main result, Theorem \ref{thm:main}.  This last step is in the spirit of the recent treatment of the incompressible Euler system by Glatt-Holtz and Vicol \cite{GHVic}. However, the analysis is more involved due to the complicated structure of \eqref{P1}--\eqref{P4}. We rely on a  delicate combination of stopping time arguments that allow to use the equivalence of \eqref{P1}--\eqref{P4} with the system studied in Section \ref{A}. As a consequence,  also the corresponding existence and uniqueness result may be applied. One of the difficulties originates in the fact that we no longer assume the initial condition to be integrable in $\omega$. Thus the a priori estimates from Section \ref{A} are no longer valid. We present the details of the proof of uniqueness in Subsection \ref{s:un}, the existence of a local strong pathwise solution in Subsections \ref{ex1} and \ref{subsec:ex2} and we conclude with the existence of a maximal strong pathwise solution in Subsection \ref{ex3}.

\section{Preliminaries and main result} \label{E}

We start by introducing the notation and some basic facts used in the text. To begin, we fix
an arbitrarily large time horizon $T>0$.

\subsection{Analytic framework}
The symbols $W^{s,p}(\mt)$ denote the Sobolov spaces of functions having distributional derivatives up to order $s$ integrable in $L^p(\mt)$ for $p\in[1,\infty]$. We will also use $W^{s,2}(\mt)$ for $s \in \R$ to denote the space of distributions $v$ defined on $\mt$ with the finite norm
\begin{equation}\label{trigo}
\left\| v \right\|^2_{W^{s,2}(\mt)} = \sum_{k \in \mathbb{Z}^N} (1 + |k|^s)^{2} |c_k(v)|^2 < \infty,
\end{equation}
where $c_k(v)$ are the Fourier coefficients of $v$ with respect to the standard trigonometric basis $\{ \exp(ik\cdot x) \}_{k \in\mathbb{Z}^N}$.
The shorten notation we will write $\|\cdot\|_{s,p}$ for $\|\cdot\|_{W^{s,p}(\mt)}$
and $\|\cdot\|_p$ for $\|\cdot\|_{L^p(\mt)}$.

The following estimates are standard in the Moser-type calculus and can be found e.g. in Majda \cite[Proposition 2.1]{Majd}.

\begin{enumerate}

\item For $u,v \in W^{s,2} \cap L^\infty(\mt)$ and $|\alpha| \leq s$
\begin{equation}\label{E5}
\left\| \partial^\alpha_x (u v) \right\|_{2} \leq c_s \left( \| u \|_{\infty} \| \nabla^s_x v \|_{2} +
\| v \|_{\infty} \| \nabla^s_x u \|_{2} \right).
\end{equation}

\item For $u \in W^{s,2}(\mt)$, $\nabla_x u \in L^\infty(\mt)$, $v \in W^{s-1,2} \cap L^\infty (\mt)$ and $|\alpha| \leq s$
\begin{equation}\label{E6}
\left\| \partial^\alpha_x (uv) - u \partial^\alpha_x v \right\|_{2} \leq c_s
\left( \| \nabla_x u \|_{\infty} \| \nabla^{s-1}_x v \|_{2} +
\| v \|_{\infty} \| \nabla^s_x u \|_{2} \right).
\end{equation}

\item For $u \in W^{s,2} \cap C(\mt)$, and $F$ $s$-times continuously differentiable function on an open neighborhood of the compact set $G =
{\rm range}[u]$, $|\alpha| \leq s$,
\begin{equation} \label{E7}
\left\| \partial^\alpha_x F(u) \right\|_{L^2(\mt)} \leq c_s \| \partial_u F \|_{C^{s-1}(G)} \| u \|^{|\alpha| - 1}_{L^\infty(\mt)} \|
\partial^\alpha_x u \|_{L^2(\mt)}.
\end{equation}
\color{black}
\end{enumerate}

\subsection{Stochastic framework}

The driving process $W$ is a cylindrical Wiener process defined on some stochastic basis $\StoB$ with a complete, right-continuous filtration, and taking values
in a separable Hilbert space $\mathfrak{U}$. More specifically, $W$ is given by a formal expansion
\[
W(t)=\sum_{k\geq 1} e_k \beta_k(t).
\]
Here $\{ \beta_k \}_{k \geq 1}$ is a family of mutually independent real-valued Brownian motions
with respect to $\StoB$ and $\{e_k\}_{k\geq 1}$ is an orthonormal basis of  $\mathfrak{U}$.
To give the precise definition of the diffusion coefficient $\mathbb{G}$, consider $\rho\in L^2(\mt)$, $\rho\geq0$, $\bfq\in L^2(\mt)$ and let {$\,\mathbb{G}(\rho,\bq):\mathfrak{U}\rightarrow L^2(\mt, \mr^N)$} be defined as follows
$$\mathbb{G}(\rho,\bq)e_k=\mathbf{G}_k(\cdot,\rho(\cdot),\bq(\cdot)).$$
We suppose that
the coefficients $\mathbf{G}_{k}:\mt\times [0,\infty) \times\mr^N\rightarrow\mr^N$ are $C^s$-functions that satisfy uniformly in $x\in\mt$
\begin{equation}\label{FG1}
\vc{G}_k (\cdot, 0 , 0) = 0,
\end{equation}
\begin{equation}
|\nabla^l \vc{G}_k (\cdot, \cdot, \cdot) | \leq \alpha_k, \quad \sum_{k \geq 1} \alpha_k  < \infty \quad \mbox{for all}\ l\in\{1,...,s\},
\label{FG2}
\end{equation}
with $s\in\mn$ specified below.
A typical example we have in mind is
\begin{align}\label{eq:model}
\mathbf{G}_k(x,\rho,\bq)=\bfa_k(x)\rho+\mathbb A_k(x) \bfq,
\end{align}
where $\bfa_k:\mt\rightarrow\R^N$ and $\mathbb A_k:\mt\rightarrow\R^{N\times N}$ are smooth functions, however, our analysis applies to  general nonlinear  coefficients $\mathbf{G}_k$.

We also introduce a new variable $r$ related to $\vr$ through formula
\[
\varrho = \varrho(r) =\left(\frac{\gamma-1}{2a\gamma}\right)^\frac{1}{\gamma-1}r^\frac{2}{\gamma-1},
\]
together with the associated family of diffusion coefficients
\[
\mathbf{F}_k(\cdot, r, \vu) = \frac{1}{\vr(r)} \mathbf{G}_k (\cdot, \vr(r), \vr(r) \vu ).
\]
Note that for the model case \eqref{eq:model} this implies
\[
\mathbf{F}_k(x,r,\bu)=\bfa_k(x)+\mathbb A_k(x)\bfu.
\]

\begin{Remark} \label{bR2}

As we are interested in \emph{strong} solutions for which both $\vr$ and $\vu$ are bounded and $\vr$ is bounded below away from zero, the hypotheses
\eqref{FG2} implies the same property for $\mathbf{F}_k$ restricted to this range. In addition, we have
\begin{align*}
\sum_k|\vc{F}_k(\cdot,r,\bfu)|\leq \,c\,(1+|\bfu|).
\end{align*}
Moreover, it is enough to assume that
\eqref{FG2} holds only locally, meaning on each compact subset of $\mt \times (0, \infty) \times \mr^N$.

\end{Remark}

\medskip

Observe that if $\vr$, $\vc{q}$ are $(\mathfrak{F}_t)$-progressively measurable $L^2(\mt)$-valued processes such that
\[
\vr \in L^2 \Big( \Omega\times[0,T]; L^2(\mt) \Big), \ \vc{q} \in L^2 \Big(\Omega\times[0,T]; L^2(\mt; \mr^N) \Big),
\]
and $\mathbb{G}$ satisfies (\ref{FG1}), (\ref{FG2}), then
the stochastic integral
\[
\int_0^t \mathbb{G}(\vr, \vr \vu) \ {\rm d} W = \sum_{k \geq 1}\int_0^t \vc{G}_k (\cdot, \vr, \vr \vu) \ {\rm d} W_k
\]
is a well-defined  $(\mathfrak{F}_t)$-martingale ranging in $L^2(\mt; \mr^N)$.

Next, we report the following result
by Flandoli and Gatarek \cite[Lemma 2.1]{FlaGat} which allows to show fractional Sobolev regularity in time for a stochastic integral.

\begin{Lemma} \label{flan}
Let $p \geq 2$, $\alpha \in[0, \frac{1}{2})$ be given. Let $\mathbb{G} = \{ \vc{G}_k \}_{k=1}^\infty$ satisfy, for some $m\in\mr$,
\[
\E \left[ \int_0^T \left( \sum_{k=1}^\infty \| \vc{G}_k \|_{W^{m,2}(\mt, \mr^N) }^2 \right)^{p/2} \dt \right] < \infty.
\]
Then
\[
t \mapsto \int_0^t \mathbb{G} \ {\rm d} W \in L^p\Big(\Omega; W^{\alpha,p} \Big(0,T;W^{m,2}(\mt; \mr^N) \Big)\Big),
\]
and there exists a constant $c = c(\alpha,p)$ such that
\[
\begin{split}
\E \left[ \left\| \int_0^t \mathbb{G}\ {\rm d} W \right\|_{W^{\alpha,p} \Big(0,T;W^{m,2}(\mt; \mr^N) \Big) }^p \right]
\leq c(\alpha,p) \E & \left[ \int_0^T \left( \sum_{k=1}^\infty \| \vc{G}_k \|_{W^{m,2}(\mt, \mr^N) }^2 \right)^{p/2}  \dt \right].
\end{split}
\]

\end{Lemma}

\begin{Remark} \label{Rflan}
Note that the above result further implies H\"older continuity of the stochastic integral due to the embedding
\[
W^{\alpha,p} \Big(0,T;W^{m,2}(\mt; \mr^N) \Big) \hookrightarrow C^\beta \Big(0,T;W^{m,2}(\mt; \mr^N) \Big) \quad \mbox{if} \quad
\beta<\alpha-\frac{1}{p}.
\]

\end{Remark}

Combining Lemma \ref{flan}, the hypotheses \eqref{FG1}, \eqref{FG2}, the estimate \eqref{E7}, and the embedding
\[
W^{s,2} (\mt) \hookrightarrow C(\mt), \ s > \frac{N}{2},
\]
we get in addition the following estimate for the stochastic integral appearing in \eqref{P2}.

\begin{Corollary} \label{Cflan}
Let $\vc{G}_k = \vc{G}_k(\vr, \vc{q})$ satisfy \eqref{FG1}, \eqref{FG2} for a nonnegative integer $s$.
Let $p\geq 2$, $\alpha \in[0, \frac{1}{2})$. Suppose that
\[
\vr, \ \vc{q} \in L^{ \beta p}\Big( \Omega \times (0,T); W^{s,2} (\mt) \Big), \ \beta = \max\{s, 1\}.
\]
Then the following holds:

{\bf (i)} If $s = 0$, then
\[
t \mapsto \int_0^t \mathbb{G}(\vr, \vc{q} ) \ {\rm d} W \in L^p\Big(\Omega; W^{\alpha,p} \Big(0,T;L^{2}(\mt; \mr^N) \Big)\Big),
\]
and
\[
\begin{split}
\E \left[ \left\| \int_0^t \mathbb{G}(\vr, \vc{q})\ {\rm d} W \right\|_{W^{\alpha,p} \Big(0,T;L^{2}(\mt; \mr^N) \Big) }^p \right]
\leq c(\alpha,p) \E & \left[ \int_0^T   \| [\vr, \vc{q}] \|_{L^{2}(\mt, \mr^N)}^p  \ \dt \right].
\end{split}
\]

{\bf (ii)} If $s >  \frac{N}{2}$,
then
\[
t \mapsto \int_0^t \mathbb{G}(\vr, \vc{q} ) \ {\rm d} W \in L^p\Big(\Omega; W^{\alpha,p} \Big(0,T;W^{s,2}(\mt; \mr^N) \Big)\Big),
\]
and
\[
\begin{split}
\E \left[ \left\| \int_0^t \mathbb{G}(\vr, \vc{q})\ {\rm d} W \right\|_{W^{\alpha,p} \Big(0,T;W^{s,2}(\mt; \mr^N) \Big) }^p \right]
\leq c(\alpha,p) \E & \left[ \int_0^T   \| [\vr, \vc{q}] \|_{W^{s, 2}(\mt, \mr^N)}^{s p}  \ \dt \right].
\end{split}
\]
\end{Corollary}

Finally, we define an auxiliary space $\mathfrak{U}_0\supset\mathfrak{U}$ via
$$\mathfrak{U}_0=\bigg\{v=\sum_{k\geq1}\alpha_k e_k;\;\sum_{k\geq1}\frac{\alpha_k^2}{k^2}<\infty\bigg\},$$
endowed with the norm
$$\|v\|^2_{\mathfrak{U}_0}=\sum_{k\geq1}\frac{\alpha_k^2}{k^2},\quad v=\sum_{k\geq1}\alpha_k e_k.$$
Note that the embedding $\mathfrak{U}\hookrightarrow\mathfrak{U}_0$ is Hilbert-Schmidt. Moreover, trajectories of $W$ are $\prst$-a.s. in $C([0,T];\mathfrak{U}_0)$.

\subsection{Main result}

Let us first introduce the notion of local strong pathwise solution. Such a solution is strong in both PDEs and probabilistic sense but possibly exists only locally in time. To be more precise, system \eqref{P1}--\eqref{P2} will be satisfied pointwise (not in the sense of distributions) on the given stochastic basis associated to the cylindrical Wiener process $W$.

\begin{Definition}[Local strong pathwise solution] \label{def:strsol}

Let $\StoB$ be a stochastic basis with a complete right-continuous filtration and let ${W}$ be an $(\mathfrak{F}_t) $-cylindrical Wiener process.
Let
$(\varrho_0,\bfu_0)$ be a $W^{s,2}(\mt)\times W^{s,2}(\mt;\mr^N)$-valued $\mathfrak{F}_0$-measurable random variable, and let $\mathbb{G}$ satisfy \eqref{FG1},
\eqref{FG2}.
A triplet
$(\varrho,\vu,\mathfrak{t})$ is called a local strong pathwise solution to system \eqref{P1}--\eqref{P4} provided
\begin{itemize}
\item $\mathfrak{t}$ is an a.s. strictly positive  $(\mathfrak{F}_t)$-stopping time;
\item the density $\varrho$ is a $W^{s,2}(\mt)$-valued $(\mathfrak{F}_t)$-progressively measurable process satisfying
$$\varrho(\cdot\wedge \mathfrak{t})  > 0,\ \varrho(\cdot\wedge \mathfrak{t}) \in C([0,T]; W^{s,2}(\tor)) \quad \mathbb{P}\text{-a.s.};$$
\item the velocity $\vu$ is a $W^{s,2}(\mt)$-valued $(\mathfrak{F}_t)$-progressively measurable process satisfying
$$ \vu(\cdot\wedge \mathfrak{t}) \in   C([0,T]; W^{s,2}(\tor; \mr^N))\cap L^2(0,T;W^{s+1,2}(\mt; \mr^N))\quad \mathbb{P}\text{-a.s.};$$
\item  there holds $\prst$-a.s.
\[
\begin{split}
\varrho (t\wedge \mathfrak{t}) &= \varrho_0 -  \int_0^{t \wedge \mathfrak{t}} \Div(\varrho\vu ) \ \dif s, \\
(\varrho \vu) (t \wedge \mathfrak{t})  &= \varrho_0 \vu_0 - \int_0^{t \wedge \mathfrak{t}} \Div (\varrho\vu \otimes\vu ) \ \dif s \\
& \qquad+ \int_0^{t \wedge \mathfrak{t}} \Div \tn{S}(\Grad \vu) \ \dif s
- \int_0^{t \wedge \mathfrak{t}}\Grad p(\varrho)\ \dif s + \int_0^{t \wedge \mathfrak{t}} {\tn{G}}(\varrho,\varrho\vu ) \ \D W,
\end{split}
\]
for all $t\in[0,T]$.
\end{itemize}
\end{Definition}

In the above definition, we have tacitly assumed that $s$ is large enough in order to provide sufficient regularity for the strong solutions.
Classical solutions require two spatial derivatives of $\vu$ to be continuous $\prst$-a.s. This motivates the following definition.

\begin{Definition}[Maximal strong pathwise solution]\label{def:maxsol}
Fix a stochastic basis with a cylindrical Wiener process and an initial condition exactly as in Definition \ref{def:strsol}. A quadruplet $$(\varrho,\vu,(\mathfrak{t}_R)_{R\in\mn},\mathfrak{t})$$ is a maximal strong pathwise solution to system \eqref{P1}--\eqref{P4} provided

\begin{itemize}
\item $\mathfrak{t}$ is an a.s. strictly positive $(\mathfrak{F}_t)$-stopping time;
\item $(\mathfrak{t}_R)_{R\in\mn}$ is an increasing sequence of $(\mathfrak{F}_t)$-stopping times such that
$\mathfrak{t}_R<\mathfrak{t}$ on the set $[\mathfrak{t}<T]$,
$\lim_{R\to\infty}\mathfrak{t}_R=\mathfrak t$ a.s. and
\begin{equation}\label{eq:blowup}
\sup_{t\in[0,\mathfrak{t}_R]}\|\vu(t)\|_{2,\infty}\geq R\quad \text{on}\quad [\mathfrak{t}<T] ;
\end{equation}
\item each triplet $(\varrho,\vu,\mathfrak{t}_R)$, $R\in\mn$,  is a local strong pathwise solution in the sense  of Definition \ref{def:strsol}.
\end{itemize}
\end{Definition}

The stopping times $\mathfrak{t}_R$ in Definition 2.6  announce the  stopping time $\mathfrak{t}$ which is therefore predictable.
It denotes the maximal life span of the solution which is   determined by the time of explosion of the $W^{2,\infty}$-norm of the velocity field.  Indeed, it can be seen from \eqref{eq:blowup} that
$$
\sup_{t\in[0,\mathfrak{t})}\|\vu(t)\|_{2,\infty}=\infty\quad \text{on}\quad [\mathfrak{t}<T].
$$
Note that the announcing sequence $(\mathfrak{t}_R)$ is not  unique. Therefore, uniqueness for maximal strong solutions is  understood in the sense that only the solution $(\vr,\vu)$ and its blow up time $\mathfrak{t}$ are unique.

Let us also point out that, later on, we will choose $s$ in order to have the embedding $W^{s,2}\hookrightarrow W^{2,\infty}$, i.e. at least $s>\frac{N}{2}+2$. Even though one might expect that the $W^{s,2}$-norm  blows up earlier than the $W^{2,\infty}$-norm, this is not true. Indeed, according to  Definition \ref{def:strsol} and Definition \ref{def:maxsol}, a maximal strong pathwise solution satisfies
$$\vu(\cdot \wedge \mathfrak{t}_R)\in C([0,T];W^{s,2}(\mt,\mr^N))\qquad\p\text{-a.s.}$$
and hence the velocity is continuous in $W^{s,2}(\mt,\mr^N)$ on $[0,\mathfrak{t})$. Consequently, the blow up of the $W^{s,2}$-norm coincides with the blow up of the $W^{2,\infty}$-norm at time $\mathfrak{t}$. This fact reflects the nature of our a priori estimates (see Subsection \ref{UNIF}): roughly speaking, control of the $W^{2,\infty}$-norm implies control of the $W^{s,2}$-norm and leads to continuity of trajectories in $W^{s,2}$.

Finally, we have all in hand to formulate our main result.

\begin{Theorem}\label{thm:main}
Let $s\in\mn$ satisfy $s>\frac{N}{2} + 3$.
Let the coefficients
$\mathbf{G}_k$ satisfy hypotheses \eqref{FG1}, \eqref{FG2} and let $(\varrho_0,\bfu_0)$ be an $\mathfrak{F}_0$-measurable, $W^{s,2}(\mt)\times W^{s,2}(\mt,\mr^N)$-valued random variable such that $\varrho_0>0$ $\p$-a.s.
Then
there exists a unique maximal strong pathwise solution $(\varrho,\vu,(\mathfrak{t}_R)_{R\in\mn},\mathfrak{t})$ to problem \eqref{P1}--\eqref{P4} with the initial condition $(\varrho_0,\vu_0)$.
\end{Theorem}

\begin{Remark} \label{mr+}

The required regularity $s > \frac{N}{2} + 3$ is definitely higher than $s > \frac{N}{2} + 2$
for the deterministic problem, see Matsumura and Nishida \cite{MANI}, \cite{MANI1}, Valli and Zajaczkowski \cite{VAZA}. This is due to the loss of regularity with respect to the time variable pertinent
to the stochastic problems. Possibly optimal results could be achieved by working in the framework of $L^p$-spaces as
Cho, Choe, and Kim \cite{ChoChoeKim} and to adapt this approach to the stochastic setting in the spirit of Glatt-Holtz and Vicol \cite{GHVic}.

\end{Remark}

\begin{Remark} \label{bR1}

The method used in the present paper can be easily adapted to handle the same problem on the
whole space $\mathcal O = \R^N$, with relevant far field conditions for $\vr$, $\vu$, say
\[
\vr \to \Ov{\vr} ,\ \vu \to 0 \ \mbox{as}\ |x| \to \infty.
\]
On the other hand, the case when the fluid interacts with a physical boundary, for instance $\mathcal O$ a bounded domain with the no-slip boundary condition for
$\vu$, would require a more elaborate treatment.

\end{Remark}

\begin{Remark}
Let us also point out that most of our analysis applies to the stochastic compressible Euler system as well. Indeed, the only point where we rely on the positive viscosity $\mu$ is the proof of continuity of trajectories of a solution in $W^{s,2}$, see Subsection \ref{subsec:ident}. It is based on the variational approach within a Gelfand triplet which gives a very elegant proof, especially in comparison to the Euler setting where one would need to find another reasoning, cf. \cite{GHVic}.
\end{Remark}

\subsection{Rewriting the equations as a symmetric hyperbolic-parabolic problem}
\label{subsec:symsyst}

It is well known in the context of compressible fluids that existence of strong solutions is intimately related to the strict positivity of the density, i.e. the non-appearance of vacuum states. Anticipating this property in the framework of strong solutions we may rewrite \eqref{P1}--\eqref{P2} as a
hyperbolic-parabolic system for unknowns $r,\bfu$ where $r$ is a function of $\varrho$.
To be more precise, as the time derivative of $\vr$ satisfies the deterministic equation (\ref{P1}), we have
\[
\D (\vr \vu) = \D \vr \ \vu + \vr\  \D \vu,
\]
where, in accordance with \eqref{P1}
\[
\D \vr  = - \Div (\vr \vu) \ \dt.
\]
Consequently, the momentum equation \eqref{P2} reads
\[
\vr \D \vu + \left[ \vr \vu \cdot \Grad \vu + a \Grad \vr^\gamma \right]  \dt  = \Div \mathbb{S} (\Grad \vu) \ \dt + \tn{G} (\vr, \vr \vu) \D W,
\]
or, anticipating strict positivity of the mass density,
\begin{equation*}
\D \vu + \left[ \vu \cdot \Grad \vu + a \frac{1}{\vr} \Grad \vr^\gamma \right]  \dt  = \frac{1}{\vr} \Div \mathbb{S} (\Grad \vu) \ \dt + \frac{1}{\vr} \tn{G} (\vr, \vr \vu) \D W.
\end{equation*}
Next, we rewrite
\[
a \frac{1}{\vr} \Grad \vr^\gamma = \frac{a \gamma}{\gamma - 1} \Grad \vr^{\gamma - 1} = \frac{2 a \gamma}{\gamma - 1}\vr^{\frac{\gamma - 1}{2}} \Grad \vr^{\frac{\gamma - 1}{2}},
\]
and evoking the renormalized variant of \eqref{P1} (cf. \cite{BrHo})
\[
\D \vr^{\frac{\gamma - 1}{2}} + \vu \cdot \Grad \vr^{\frac{\gamma - 1}{2}}\dt + \frac{\gamma - 1}{2} \vr^{\frac{\gamma - 1}{2}} \Div \vu \dt = 0.
\]
Thus, for a new variable
\[
r \equiv \sqrt{ \frac{2 a \gamma}{\gamma - 1} } \vr^{\frac{\gamma - 1}{2}},
\]
system \eqref{P1}, \eqref{P2} takes the form
\begin{equation} \label{E3}
\D r + \vu \cdot \Grad r \ \dt + \frac{\gamma - 1}{2} r \Div \vu\dt = 0,
\end{equation}
\begin{equation} \label{E4}
\D \vu + \left[ \vu \cdot \Grad \vu + r \Grad r \right]  \dt  = D(r) \Div \mathbb{S} (\Grad \vu) \dt + \tn{F} (r, \vu) \D W,
\end{equation}
where
$$D(r)=\frac{1}{\varrho}=\left(\frac{\gamma-1}{2a\gamma}\right)^{-\frac{1}{\gamma-1}}r^{-\frac{2}{\gamma-1}}, \quad \tn{F} (r, \vu) = \frac{1}{\vr(r)} \tn{G} (\vr(r), \vr(r) \vu).$$
Observe that the left hand side corresponds to a symmetric hyperbolic system, cf. Majda \cite{Majd}, for which higher order energy estimates can be obtained by differentiating \eqref{E3}, \eqref{E4} in $x$ up to order $s$,
cf. Gallagher \cite{Gall2000}, Majda \cite{Majd}. Unlike the more elaborated treatment proposed by Cho, Choe, and Kim \cite{ChoChoeKim}
giving rise to the optimal regularity space for the deterministic
Navier-Stokes system, the energy approach avoids differentiating the equations in the time variable - a procedure that may be delicate in the stochastic setting.

\subsection{Outline of the proof of Theorem \ref{thm:main}}
\label{subsec:outline}

 In the \emph{deterministic} setting, system \eqref{E3}--\eqref{E4} can be solved via an approximation procedure. The so-obtained local in time strong
 solution exists on a maximal time interval, the length of which can be estimated in terms of the
size of the initial data. However, in the \emph{stochastic} setting it is more convenient to work with approximate solutions defined on the whole time interval $[0,T]$. To this end, we introduce suitable cut-off operators applied to the $W^{2,\infty}$-norm of the velocity field.
Specifically, we consider
the approximate system in the form
\begin{align} \label{E3'}
\D r + \varphi_R (\|\vu \|_{2,\infty})\Big[\vu \cdot \Grad r \  + \tfrac{\gamma - 1}{2}\, r\, \Div \vu\Big]\ \dt &= 0,\\
\label{E4'}
\D \vu + \varphi_R(\| \vu \|_{2,\infty})\left[ \vu \cdot \Grad \vu + r \Grad r \right]  \dt & =\varphi_R(\| \vu \|_{2
,\infty}) D(r) \Div \mathbb{S} (\Grad \vu) \ \dt\\& +  \varphi_R(\| \vu \|_{2,\infty})\tn{F} (r, \vu) \D W\nonumber,\\
r(0)=r_{0} ,\quad \bfu(0)&=\bfu_0,\label{approx:initial}
\end{align}
where $\varphi_R:[0,\infty)\rightarrow[0,1]$ are smooth cut-off functions satisfying
\begin{align*}
\varphi_R(y)=\begin{cases}1,\quad &0\leq y\leq R,\\
0,\quad & R+1\leq y.
\end{cases}
\end{align*}

Our aim is to solve \eqref{E3'}--\eqref{approx:initial} via the stochastic compactness method: First, we construct solutions to certain approximated systems, establish tightness of their laws in suitable topologies and finally deduce the existence of a strong martingale solution to \eqref{E3}--\eqref{E4}. The necessary uniform bounds are obtained through
a purely hyperbolic approach by differentiating with respect to the space variable and testing the resulting expression with
suitable space derivative of the unknown functions.

For the above mentioned reasons, the approximated densities must be positive on time intervals of finite length. Therefore the approximation scheme
must be chosen to preserve the maximum principle for (\ref{E3'}). To this end,
the approximate solutions to \eqref{E3'}--\eqref{approx:initial} will be constructed by means of a hybrid method based on
\begin{itemize}
\item
solving the (deterministic) equation of continuity \eqref{E3'} for a given $\vu$ obtaining $r = r[\vu]$;
\item
plugging $r = r[\vu]$ in \eqref{E4'} and using a fixed point argument to get local in time solutions of a Galerkin approximation of \eqref{E4'};
\item
extending the Galerkin solution to $[0,T]$ by means of {\it a priori} bounds.
\end{itemize}

Note that the transport equation \eqref{E3'} is solved exactly in terms of a given velocity field $\vu$
as the cut-off operators apply to $\vu$ only.

\section{The approximated system}
\label{A}

In this section we focus on the approximated system  \eqref{E3'}--\eqref{E4'}. More precisely, our aim is twofold: First, we establish existence of a strong martingale solution for initial data in $L^p(\Omega;W^{s,2}(\mt))$ for all $1\leq p<\infty$ and some $s>\frac{N}{2}+2$; second, we prove pathwise uniqueness provided $s>\frac{N}{2}+3$, which in turn implies existence of a (unique) strong pathwise solution.

To this end, let us introduce these two concepts of strong  solution
for the approximate system \eqref{E3'}--\eqref{E4'}. A strong martingale solution is strong in the PDEs sense but only weak in the probabilistic sense. In other words, the stochastic basis as well as a cylindrical Wiener process cannot be given in advance and become a part of the solution. Accordingly, the initial condition is stated in the form of a initial law. On the other hand, a strong pathwise solution is strong in both PDEs and probabilistic sense, that is, the stochastic elements are given in advance.

\begin{Definition}[Strong martingale solution] \label{def:strsolmart}
Let $\Lambda$ be a Borel probability measure on
$$W^{s,2}(\mt)\times W^{s,2}(\mt,\mr^N).$$
A multiplet
$$\left(\StoB,r,\vu,W\right)$$
is called a strong martingale solution to the approximate system \eqref{E3'}--\eqref{E4'} with the initial law $\Lambda$, provided
\begin{itemize}
\item $\StoB$ is a stochastic basis with a complete right-continuous filtration;
\item ${W}$ is an $( \mathfrak{F}_t ) $-cylindrical Wiener process;
\item $r$ is a $W^{s,2}(\mt)$-valued $(\mathfrak{F}_t)$-progressively measurable process satisfying
$$r \in L^2 \Big(\Omega; C([0,T]; W^{s,2}(\tor)) \Big);$$

\item the velocity $\vu$ is a $W^{s,2}(\mt)$-valued $(\mathfrak{F}_t)$-progressively measurable process satisfying
$$ \vu \in L^2\Big( \Omega; C([0,T]; W^{s,2}(\tor))\cap L^2(0,T;W^{s+1,2}(\mt)) \Big);$$
\item $\Lambda=\p\circ[(r(0),\vu(0))]^{-1};$

\item there holds $\prst$-a.s.
\[
\begin{split}
r(t) &= r({0}) -  \int_0^{t} \varphi_R (\|\vu \|_{2,\infty})\Big[\vu \cdot \Grad r \  + \tfrac{\gamma - 1}{2}\, r\, \Div \vu\Big]\ {\rm d}s, \\
\vu (t)  &= \vu(0) - \int_0^{t} \varphi_R(\| \vu \|_{2,\infty})\left[ \vu \cdot \Grad \vu + r \Grad r \right] \dif s \\
& \qquad+ \int_0^{t} \varphi_R(\| \vu \|_{2
,\infty}) D(r) \Div \mathbb{S} (\Grad \vu)  \dif s
 + \int_0^{t} \varphi_R(\| \vu \|_{2,\infty})\tn{F} (r, \vu) \ \D W,
\end{split}
\]
for all $t\in[0,T]$.
\end{itemize}
\end{Definition}

\begin{Definition}[Strong pathwise solution] \label{def:strsolpath}
Let $\StoB$ be a given stochastic basis with a complete right-continuous filtration and let ${W}$ be a given $( \mathfrak{F}_t ) $-cylindrical Wiener process.
Then $(r,\vu)$
is called a strong pathwise solution to the approximate system \eqref{E3'}--\eqref{E4'} with the initial condition $(r_0,\vu_0)$ provided
\begin{itemize}
\item $r$ is a $W^{s,2}(\mt)$-valued $(\mathfrak{F}_t)$-progressively measurable process satisfying
$$r \in L^2 \Big(\Omega; C([0,T]; W^{s,2}(\tor)) \Big);$$

\item the velocity $\vu$ is a $W^{s,2}(\mt)$-valued $(\mathfrak{F}_t)$-progressively measurable process satisfying
$$ \vu \in L^2\Big( \Omega; C([0,T]; W^{s,2}(\tor))\cap L^2(0,T;W^{s+1,2}(\mt)) \Big);$$

\item there holds $\prst$-a.s.
\[
\begin{split}
r(t) &= r_{0} -  \int_0^{t} \varphi_R (\|\vu \|_{2,\infty})\Big[\vu \cdot \Grad r \  + \tfrac{\gamma - 1}{2}\, r\, \Div \vu\Big]\ {\rm d}s, \\
\vu (t)  &= \vu_0 - \int_0^{t} \varphi_R(\| \vu \|_{2,\infty})\left[ \vu \cdot \Grad \vu + r \Grad r \right] \dif s \\
& \qquad+ \int_0^{t} \varphi_R(\| \vu \|_{2
,\infty}) D(r) \Div \mathbb{S} (\Grad \vu)  \dif s
 + \int_0^{t} \varphi_R(\| \vu \|_{2,\infty})\tn{F} (r, \vu) \ \D W,
\end{split}
\]
for all $t\in[0,T]$.
\end{itemize}
\end{Definition}

The main result of this section reads as follows.

\begin{Theorem}\label{thm:appr}
Let the coefficients
$\mathbf{G}_k$ satisfy hypotheses \eqref{FG1}, \eqref{FG2} and let
\[
(r_0,\bfu_0)\in L^p(\Omega,\mathfrak{F}_0,\p;W^{s,2}(\mt)\times W^{s,2}(\mt))
\]
for all $1\leq p<\infty$ and some $s\in\mn$ such that $s>\frac{N}{2} + 2$. In addition, suppose that
\begin{equation*}
\| r_0 \|_{W^{1,\infty}(\mt)} < R, \ r_0 > \frac{1}{R} \ \prst\mbox{-a.s.}
\end{equation*}
Then there exists a strong martingale solution
to problem \eqref{E3'}--\eqref{E4'} with the initial law $\Lambda=\p\circ[(r_{0},\vu_0)]^{-1}$. Moreover,  there exists a deterministic constant $\underline{r}_R>0$ such that
\[
r(t, \cdot) \geq \underline{r}_R > 0 \quad \mathbb{P}\mbox{-a.s.}\quad \mbox{for all}\ t \in [0,T]
\]
and
\begin{equation} \label{RE2}
\E\bigg[\sup_{t \in [0,T]} \| (r(t),\vu(t)) \|_{s,2} +\int_0^T \| \vu \|^2_{s+1,2} \ \dt  \bigg]^p \leq
c(R,r_0, \vu_0,p) < \infty \quad \mbox{for all}\quad 1 \leq p
< \infty.
\end{equation}

Finally, if $s > \frac{N}{2} + 3$, then  pathwise uniqueness holds true. Specifically, if  $(r^1, \vu^1)$, $(r^2, \vu^2)$ are two strong solutions to \eqref{E3'}--\eqref{E4'} defined on the same stochastic basis with the same Wiener process $W$ and
$$\mathbb{P} \left[ r^1_0  = r^2_0, \ \vu^1_0 = \vu^2_0 \right] = 1,$$
then
\[
\mathbb{P} \left[ r^1 (t) = r^2(t), \ \vu^1(t) = \vu^2(t),  \ \mbox{for all}\ t \in [0,T]\right] = 1.
\]
Consequently, there exists a unique strong pathwise solution to \eqref{E3'}--\eqref{E4'}.

\end{Theorem}

The rest of this section is dedicated to the proof of Theorem \ref{thm:appr} which is divided into several parts. First, in Subsection \ref{subsec:galerkin} we construct the approximate solutions to \eqref{E3'}--\eqref{E4'} by employing the hybrid method delineated in Subsection \ref{subsec:outline}. Second, in Subsection \ref{UNIF} we derive higher order energy estimates which hold true uniformly in the approximation parameter $n$. Third, in Subsection \ref{subsec:comp} we perform the stochastic compactness method: we establish tightness of the laws of the approximated solutions and apply the Skorokhod representation theorem. This yields existence of a new probability space with a sequence of random variables converging a.s. Then in Subsection \ref{subsec:ident}, we identify the limit with a strong martingale solution to \eqref{E3'}--\eqref{E4'}. Finally, in Subsection \ref{subsec:uniq} we provide the proof of pathwise uniqueness under the additional assumption that $s>\frac{N}{2}+3$ and in Subsection \ref{subsec:pathwiseexist} we employ the Gy\"ongy-Krylov argument to deduce the existence of a strong pathwise solution.

\subsection{The Galerkin approximation}
\label{subsec:galerkin}

To begin with, observe that for any $\vu \in C([0,T]; W^{2,\infty}(\mt))$, the transport equation (\ref{E3'}) admits a classical solution
$r = r[\vu]$, uniquely determined by the initial datum $r_{0}$. In addition, for a certain universal constant $c$ we have the estimates
\begin{equation} \label{est1}
\begin{split}
\frac{1}{R}
\exp \left( - cR t \right)&\leq
\exp \left( - cR t \right)  \inf_{\mt} r_0  \leq r(t, \cdot) \leq \exp \left(  cR t \right) \sup_{\mt} r_{0} \leq  R \exp \left(  cR t \right) \\
|\Grad r (t,\cdot) | & \leq \exp \left(  cR t \right)|\Grad r_0 |\leq
R \exp \left(  cR t \right) \ t \in [0,T].
\end{split}
\end{equation}

Next, we consider the orthonormal basis $\left\{ \bfpsi_m \right\}_{m=1}^\infty$ of the space $L^{2}(\mt; \R^N)$ formed by trigonometric functions and set
\[
X_n = {\rm span} \left\{ \bfpsi_1, \dots, \bfpsi_n \right\}, \quad \mbox{with the associated projection}\ P_n : L^2 \to X_n.
\]
We look
for approximate solutions $\vu_n$ of \eqref{E4'} belonging to {$L^2\Big( \Omega; C([0,T]; X_n) \Big)$},
satisfying
\begin{equation} \label{est2}
\begin{split}
\D\left< \vu_n, \bfpsi_i \right> &+ \varphi_R(\| \vu_n \|_{2,\infty}) \left< \Big[[\vu_n \cdot \Grad \vu_n + r[\vu_n] \Grad r[\vu_n, r_{0,R}] \Big]; \bfpsi_i \right>  \dt
\\
&=\varphi_R(\| \vu_n \|_{2,\infty}) \left< D(r[\vu_n]) \Div \mathbb{S} (\Grad \vu_n); \bfpsi_i \right> \ \dt\\  &
+  \varphi_R(\| \vu_n \|_{2,\infty}) \left< \tn{F} (r[\vu_n] , \vu_n); \bfpsi_i \right> \D W,\quad i = 1, \dots, n.\\
\vu_n(0)&=P_n \bfu_0.
\end{split}
\end{equation}
As all norms on $X_n$ are equivalent, solutions of \eqref{E3'}, \eqref{est2} can be obtained in a standard way by means of the Banach fixed point argument.
Specifically, we have to show that the mapping
\[
\vu \mapsto \mathscr{T} \vu : X_n \to X_n,
\]
\begin{equation} \label{est3}
\begin{split}
\left< \mathscr{T}\vu; \psi_i \right> =& \left< \vu_0; \psi_i \right>- \int_0^\cdot\varphi_R(\| \vu \|_{2,\infty}) \left< \Big[\vu\cdot \Grad \vu + r[\vu] \Grad r[\vu, r_{0,n} ] \Big]; \bfpsi_i \right>  \dt
\\
&+\int_0^\cdot\varphi_R(\| \vu \|_{2,\infty}) \left< D(r[\vu]) \Div \mathbb{S} (\Grad \vu); \bfpsi_i \right> \ \dt\\  &
+  \int_0^\cdot\varphi_R(\| \vu \|_{2,\infty}) \left< \tn{F} (r[\vu] , \vu); \bfpsi_i \right> \D W,\quad i = 1, \dots, n.
\end{split}
\end{equation}
is a contraction on $\mathcal B=L^2(\Omega;C([0,T^\ast]; X_n))$ for $T^\ast$ sufficiently small. The three components of $\mathscr{T}$ appearing on the right hand side of \eqref{est3}
will be denoted by $\mathscr T_{det}^1$, $\mathscr T_{det}^2$ and $\mathscr T_{sto}$, respectively.
\color{black}

For $r_1 = r[\vu]$, $r_2 = r[\vc{v}]$, we get
\begin{equation*} \label{T8}
\begin{split}
\D (r_1 - r_2) &+ \vc{v}_1 \cdot \Grad (r_1 - r_2) \dt - \frac{\gamma - 1}{2} \Div \vc{v}_1 (r_1 - r_2) \dt \\
&= - \Grad r_2 \cdot (\vc{v}_1 - \vc{v}_2) - \frac{\gamma - 1}{2}  r_2 \Div (\vc{v}_1 - \vc{v}_2) \dt,
\end{split}
\end{equation*}
where we have set
\[
\vc{v}_1 = \varphi_R(\| \vu \|_{2,\infty}) \vu , \ \vc{v}_2 = \varphi_R(\| \vc{v} \|_{2,\infty})\vc{v}.
\]
Consequently, we easily deduce that
\begin{align}\label{eq:new}
\sup_{0\leq t\leq T^\ast}\big\|r[\bfu]-r[\bfv]\big\|^2_{L^2}\leq T^\ast C(n,R,T)\sup_{0\leq t\leq T^\ast} \big\|\bfu-\bfv\big\|_{X_n}^2
\end{align}
noting that $r_1$, $r_2$ coincide at $t=0$ and that $r_j$, $\Grad r_j$ are bounded by a deterministic constant depending on $R$.

As a consequence of \eqref{est1}, \eqref{eq:new} and the equivalence of norms on $X_n$ we can show that the mapping
$\mathscr T_{det}=\mathscr T_{det}^1+\mathscr T_{det}^2$ satisfies the estimate
\begin{align}\label{Tdet}
\|\mathscr{T}_{det}\bfu-\mathscr{T}_{det}\bfv\|_\mathcal{B}^2\leq T^{\ast}C(n,R,T)\|\bfu-\bfv\|_{\mathcal B}^2.
\end{align}
Finally, by Burgholder-Davis-Gundy inequality we have (setting $J_R(\bfw)=\varphi_{R+1}(\|\bfw\|_{2,\infty})\bfw $)
\begin{equation*}
\begin{split}
\|\mathscr{T}_{sto}\bfu&-\mathscr{T}_{sto}\bfv\|_\mathcal{B}^2=\,\stred\sup_{0\leq t\leq T_\ast}\bigg\|\int_0^t\Big(\varphi_R(\|\vu\|_{2,\infty})\mathbb F\big(r[\vu],\vu\big)-\varphi(\|\vv\|_{2,\infty})\mathbb F\big(r[\vv],\vv\big)\Big)\,\dif W\bigg\|_{X_n}^2\\
&\leq C(n,R)\,\stred\int_0^{T^{\ast}}\sum_{k\geq1}\Big\|\,\varphi_R(\|\vu\|_{2,\infty})\bfF_k\big(r[\vu],J_R(\vu)\big)-\varphi_R(\|\vv\|_{2,\infty})\mathbb \bfF_k\big(r[\vv],J_R(\vv)\big)\Big\|_{X_n}^2\dif s\\
&\leq C(n,R)\,\stred\int_0^{T^{\ast}}\big|\varphi_R(\|\vu\|_{2,\infty})-\varphi_R(\|\vv\|_{2,\infty})\big|^2\sum_{k\geq1}\Big\|\,\bfF_k\big(r[\vu],J_R(\vu)\big)\Big\|_{X_n}^2\dif s\\
&+ C(n,R)\,\stred\int_0^{T^{\ast}}\varphi_R(\|\vv\|_{2,\infty})^2\sum_{k\geq1}\Big\|\, \bfF_k\big(r[\vu],J_R(\vu)\big)-\mathbb \bfF_k\big(r[\vv],J_R(\vv)\big)\Big\|_{X_n}^2\dif s.
\end{split}
\end{equation*}
Using the growth conditions for $\bfF_k$ (see \eqref{FG2} and Remark \ref{bR2})
we gain
\begin{align}
\|&\mathscr{T}_{sto}\bfu-\mathscr{T}_{sto}\bfv\|_\mathcal{B}^2\nonumber
\\&\leq T^\ast C(n,R)\bigg(\stred\|\bfu-\bfv\|_{2,\infty}+\,\stred\int_0^{T_{\ast}}\big\|\,r[\vu]-r[\vv]\big\|_{L^2}^2\dif s+\,\stred\int_0^{T_{\ast}}\big\|J_R(\vu)-J_R(\vv)\big\|_{L^2}^2\dif s\bigg)\nonumber\\
&\leq T^{\ast}C(n,R)\|\bfu-\bfv\|_{\mathcal B}^2.\label{Tsto}
\end{align}
Note that the last step was a consequence of \eqref{eq:new} and the equivalence of norms. Combining \eqref{Tdet} and \eqref{Tsto} shows that
$\mathscr T$ is a contraction for {a deterministic (small) time $T^\ast> 0$}. A solution to \eqref{E3'}--\eqref{E4'} on the whole interval $[0,T]$ can be obtained by decomposing it into small subintervals gluing the corresponding solutions together.

\subsection{Uniform estimates}
\label{UNIF}

In this subsection, we derive estimates that hold uniformly for $n \to \infty$, which yields a basis for our compactness argument presented in Subsection \ref{subsec:comp}.
{At this stage, the approximate velocity field $\vu_n$ is smooth in the $x$-variable; whence the corresponding solution
$r_n = r[\vu_n, r_{0,n}]$ of the transport equation (\ref{E3'}) shares the same smoothness with the initial datum $r_0$.}

Let $\alpha$ be a multiindex such that $|\alpha|\leq s$. Differentiating \eqref{E3'} in the $x$-variable we obtain
\begin{align} \label{E8'}
\begin{aligned}
\D \partial^\alpha_x r_n &+ \varphi_R(\| \vu_n \|_{2,\infty})\left[\vu_n \cdot \Grad \partial^\alpha_x r_n \ + \tfrac{\gamma - 1}{2} \,r_n\, \Div \partial^\alpha_x  \vu_n \right]\ \dt  \\
&=\varphi_R(\| \vu_n \|_{2,\infty})\big[ \vu_n \cdot \partial^\alpha_x \Grad  r_n  - \partial^\alpha_x \left( \vu_n \cdot \Grad  r_n \right) \big] \dt\\
&+ \tfrac{\gamma - 1}{2}\varphi_R(\| \vu_n \|_{2,\infty})
\left[r_n \partial^\alpha_x \Div  \vu_n - \partial^\alpha_x \left( r_n  \Div  \vu_n \right) \right] \dt\\
&=:T_1^n \dt+T^n_2 \dt.
\end{aligned}
\end{align}
Similarly, we may use the fact that the spaces $X_n$ are invariant with respect to the spatial derivatives, in particular, we deduce
that
\begin{align} \label{E9'}
\begin{aligned}
\D \left< \partial^\alpha_x \vu_n; \bfpsi_i \right>  &+ \varphi_R(\| \vu_n \|_{2, \infty})\ \left< \left[ \vu_n \cdot \Grad \partial^\alpha_x \vu_n + r_n \Grad \partial^\alpha_x r_n \right] ; \bfpsi_i \right>  \dt\\
&  -  \varphi_R(\| \vu_n \|_{2,\infty}) \left< D(r_n) \Div \mathbb{S} (\Grad \partial^\alpha_x \vu_n); \bfpsi_i \right> \ \dt \\
  &= \varphi_R(\| \vu_n \|_{2,\infty})\left< \left[ \vu_n \cdot \partial^\alpha_x \Grad  \vu_n - \partial^\alpha_x \left( \vu_n \cdot \Grad  \vu_n \right)  \right];
  \bfpsi_i \right>  \dt  \\
&+ \varphi_R(\| \vu_n \|_{2,\infty}) \left< \left[ r_n \partial^\alpha_x \Grad  r_n -
\partial^\alpha_x \left( r_n  \Grad  r_n \right) \right]; \bfpsi_i \right>  \dt  \\
&-  \varphi_R(\| \vu_n \|_{2, \infty}) \left< \left[ D(r_n) \partial^\alpha_x \Div \mathbb{S} (\Grad  \vu_n)
-  \partial^\alpha_x \left( D(r_n) \Div \mathbb{S} (\Grad \vu_n) \right)
 \right]; \bfpsi_i \right> \dt \\& +  \varphi_R(\| \vu_n \|_{2,\infty}) \left< \partial^\alpha_x \tn{F} (r_n, \vu_n); \bfpsi_i \right> \D W\\
&=:T_3^n\dt+T^n_4 \dt+T_5^n\dt+  \varphi_R(\| \vu_n \|_{2, \infty}) \left< \partial^\alpha_x \tn{F} (r_n, \vu_n); \bfpsi_i \right> \D W,\quad i = 1,\dots,n.
\end{aligned}
\end{align}
It follows from \eqref{E6} that the ``error'' terms may be handled as
\begin{equation} \label{E10'}
\begin{split}
\left\| T_1^n\right\|_{2} & \lesssim \varphi_R(\| \vu_n \|_{2,\infty}) \Big[
\| \Grad \vu_n \|_{\infty} \| \Grad^s r_n \|_{2} + \left\| \Grad r_n \right\|_{\infty} \| \Grad^s \vu_n \|_{2} \Big] \\
\left\| T_2^n \right\|_{2} & \lesssim \varphi_R(\| \vu_n \|_{2,\infty}) \Big[
\| \Grad r_n \|_{\infty} \| \Grad^s \vu_n  \|_{2} + \left\| \Div \vu_n  \right\|_{\infty} \| \Grad^s r_n \|_{2} \Big]
\\
\left\|  T_3^n \right\|_{2} & \lesssim \varphi_R(\| \vu_n \|_{2,\infty})
\| \Grad \vu_n \|_{\infty} \| \Grad^s \vu_n  \|_{2}  \\
\left\|  T_4^n \right\|_{2} & \lesssim
\| \Grad r_n \|_{\infty} \| \Grad^s r  _n\|_{2},
\end{split}
\end{equation}
and
\begin{equation} \label{E11'}
\begin{split}
\left\| T_5^n \right\|_{2} & \lesssim \varphi_R(\| \vu_n \|_{2,\infty})
\left\| \Grad D(r_n) \right\|_{\infty} \left\| \Grad^s \mathbb{S} (\Grad \vu_n) \right\|_{2} \\
&\quad +\varphi_R(\| \vu_n \|_{2,\infty})
\left\| \Div \mathbb{S} (\Grad \vu_n) \right\|_{\infty} \left\| \Grad^s D(r_n) \right\|_{2} .
\end{split}
\end{equation}
Multiplying \eqref{E8'} by $\partial^\alpha_x r_n$ and integrating the resulting expression by parts, we observe
$$
\int_{\mt}\vu_n \cdot \Grad \partial^\alpha_x r \partial^\alpha_x r_n\,\dx=-\frac{1}{2}\int_{\mt}\Div\vu_n|\partial^\alpha_x r_n|^2\,\dx;
$$
whence
\begin{align} \label{E12'}
\begin{aligned}
&\left\| \partial^\alpha_x r_n (t) \right\|_{2}^2 + (\gamma - 1)\int_0^t\varphi_R(\| \vu_n \|_{2,\infty}) \int_{\mt} r_n \Div \partial^\alpha_x \vu_n \partial^\alpha_x r_n \dx\ds \\
&\quad\lesssim \left\| \partial^\alpha_x r_0 \right\|_{2}^2 +\int_0^t\varphi_R(\| \vu_n \|_{2, \infty})
\left( \| \vu_n \|_{1, \infty} \| \vr \|_{s,2} + \| r_n \|_{1,\infty} \| \vu _n\|_{s,2} \right)
\|\partial^\alpha_x r_n\|_2\,\ds\\
\end{aligned}
\end{align}
provided $|\alpha| \leq s$.

To apply the same treatment to \eqref{E9'}, we use It\^o's formula for the function $f(\bfC^n)=\int_{\mt}|\partial^\alpha_x\bfu_n|^2\dx$. There holds
\begin{equation} \label{E13'}
\begin{split}
 \left\|\partial^\alpha_x \vu_n(t) \right\|^2_2 \dx &+2\int_0^t\varphi_R(\| \vu_n \|_{2,\infty})\int_{\mt} \left[ \vu_n \cdot \Grad \partial^\alpha_x \vu_n + r_n \Grad \partial^\alpha_x r_n  \right] \cdot \partial^\alpha_x \vu_n \dx\ds \\
 &-2\int_0^t\varphi_R(\| \vu_n\|_{2, \infty})\int_{\mt}  D(r_n) \Div \mathbb{S} (\Grad \partial^\alpha_x \vu_n) \cdot \partial^\alpha_x \vu_n \dx\ds \\
&=  \left\| \partial^\alpha_x P_n\vu_0 \right\|^2  + 2\int_0^t\int_{\mt}\left[ T_3^n+T_4^n+T_5^n\right] \cdot \partial^\alpha_x \vu_n  \dx\ds  \\
& + 2\int_0^t \varphi_R(\| \vu_n \|_{2, \infty})\int_{\mt} \partial^\alpha_x \tn{F} (r_n, \vu_n) \cdot \partial^\alpha_x \vu_n \ \D W\\
&+\sum_{k\geq1} \int_0^t \varphi_R(\| \vu_n \|_{2,\infty})\int_{\mt}| \partial^\alpha_x \vc{F}_k (r_n, \vu_n)|^2 \dx\ds.
\end{split}
\end{equation}
Integrating by parts yields
\[
\begin{split}
\int_{\mt} &\big[ \vu_n \cdot \Grad \partial^\alpha_x \vu_n + r _n \Grad \partial^\alpha_x r _n  \big] \cdot \partial^\alpha_x \vu_n \dx\\&= - \frac{1}{2} \int_{\mt} |\partial^\alpha_x \vu_n |^2 \Div \vu_n \dx
 - \int_{\mt}r_n \Div \partial^\alpha_x \vu_n \partial^\alpha_x r_n \dx  - \int_{\mt} \Grad r_n \cdot \partial^\alpha \vu_n \partial^\alpha_x r_n
\end{split}
\]
as well as
\[
\begin{split}
-\int_{\mt} &\big[ D(r _n) \Div \mathbb{S} (\Grad \partial^\alpha_x \vu_n) \big] \cdot \partial^\alpha_x \vu_n \dx\\&=
 \int_{\mt} \Grad D(r _n) \cdot \mathbb{S} (\Grad \partial^\alpha_x \vu_n) \cdot \partial^\alpha_x \vu_n \dx
 + \int_{\mt} D(r_n) \mathbb{S} (\Grad \partial^\alpha_x \vu_n ) : \Grad \partial^\alpha_x \vu_n \dx
\end{split}
\]

Summing up \eqref{E12'}--\eqref{E13'} and using \eqref{E10'}--\eqref{E11'} we observe that the term containing $r_n\partial^\alpha_x r_n\Div \partial^\alpha_x \vu_n$ on the left hand side cancels out and we may infer that
\begin{align*}
&\left\| (r_n(t),\bfu_n(t)) \right\|^2_{s,2}  + \sum_{|\alpha| \leq s } \int_0^t\int_{\mt}\varphi_R(\| \vu_n \|_{2,\infty}) D(r_n) \mathbb{S} (\Grad \partial^\alpha_x \vu_n ) : \Grad \partial^\alpha_x \vu_n \dx \ds \\
&\quad \lesssim  \left\| (r_0 ,\bfu_0) \right\|^2_{s,2}+
\int_0^t \left[ \varphi_R(\| \vu_n \|_{2, \infty})
\| \vu_n \|_{1, \infty} \Big( \| r_n \|^2_{s,2} + \| \vu \|^2_{s,2} \Big)  + \| r_n \|_{1,\infty} \| r_n \|_{s,2} \| \vu _n\|_{s,2} \right] \dt
 \\
&\quad + \int_0^t \left[ \varphi_R(\| \vu_n \|_{2,\infty})\| \Div \mathbb{S} (\Grad \vu_n ) \|_{\infty} \left\| D(r_n) \right\|_{s,2} \| \vu_n \|_{s,2}
   + \| \Grad D(r_n) \|_{\infty} \| \vu_n \|^2_{s,2}  \right] \ds\\
&\quad +  \int_0^t \varphi_R(\| \vu_n\|_{2,\infty})\int_{\mt}  \partial^\alpha_x \tn{F} (r_n, \vu_n) \cdot \partial^\alpha_x \vu_n \dx\ \D W \\
&\quad +\sum_{k\geq1}\int_0^t \varphi_R(\| \vu_n \|_{2,\infty})\int_{\mt} | \partial^\alpha_x \vc{F}_k (r_n, \vu_n)|^2 \dx \ \ds
\end{align*}
as long as $s > \frac{N}{2} + 2$.

\begin{Remark} \label{estR1}

Note that the above estimate {depends on $R$ only through the cut-off function $\varphi_R$}. Moreover, in accordance with \eqref{est1},
\[
\begin{split}
 \varphi_R(\| \vu_n \|_{2, \infty})
\| \vu_n \|_{1, \infty} + \varphi_R(\| \vu_n \|_{2,\infty})\| \Div \mathbb{S} (\Grad \vu_n ) \|_{\infty} &\lesssim cR, \\
\ \| r_n \|_{1,\infty} + \| \Grad D(r_n) \|_{\infty} &\lesssim c \exp \left( cRT \right) \Big( \| r_{0} \|_{1,\infty} + \| \Grad r_{0} \|_{\infty} \Big) \\
& \lesssim c(R) \exp \left( cRT \right),
\end{split}
\]
and, in view of \eqref{E7}, \eqref{est1},
\[
\left\| D(r_n) \right\|_{s,2} \leq c(R,T) \| r_n \|_{s,2}.
\]

\end{Remark}

In contrast with the preceding part,
the following inequalities \emph{depend on $R$}. Using \eqref{FG2} as well as \eqref{E7} we have
\begin{align*}
\sum_{k\geq1}\int_0^t&\varphi_R(\| \bfu_n \|_{2,\infty})\int_{\mt} |\partial^\alpha_x \vc{F}_k (r_n, \vu_n)|^2 \dx \ \ds\\
&\lesssim \int_0^t \varphi_R(\| \bfu_n \|_{2,\infty}) \int_{\mt}\sum_{k\geq1} |\nabla^{s-1} \vc{F}_k|^2 \,\dx \,\|(r_n,\vu_n)\|_{\infty}^{2(|\alpha|-1)}\|(r_n,\vu_n)\|_{s,2}^2\,\ds \\
&\lesssim c(R,T) \int_0^t \|(r_n,\vu_n)\|_{s,2}^2 \,\ds.
\end{align*}
as well as
\begin{align*}
\E\bigg[\sup_{t\in(0,T)}&\bigg|\int_0^t \varphi_R(\|\vu_n \|_{2,\infty})\int_{\mt} \partial^\alpha_x \tn{F} (r_n, \vu_n) \cdot \partial^\alpha_x \vu_n\,\dx \ \D W\bigg|\bigg]^p\\
&\lesssim
\E\bigg[\sum_{k\geq1}\int_0^T \varphi_R(\| \vu_n \|_{2,\infty})^2\bigg(\int_{\mt} \partial^\alpha_x \vc{F}_k (r_n, \vu_n) \cdot \partial^\alpha_x \vu_n \dx\bigg)^2\dx\dt\bigg]^{\frac{p}{2}}\\
&\lesssim
\E\bigg[\int_0^T \varphi_R(\| \vu_n \|_{2,\infty})^2\bigg(\sum_{k\geq1}\|\vc{F}_k (r_n, \vu_n)\|^2_{s,2}\bigg) \|\vu_n\|^2_{s,2}\dt\bigg]^{\frac{p}{2}}\\
&\lesssim \E\bigg[\int_0^T \varphi_R(\|\vu_n\|_{2,\infty})^2 \|(r_n,\vu_n)\|_\infty^{2(s-1)}\|(r_n,\vu_n)\|_{s,2}^4\,\ds\bigg]^\frac{p}{2}\\
&\lesssim c(R,T)  \E\bigg[\sup_{t\in(0,T)}\| \vu_n\|_{s,2}^2\int_0^T\|(r_n,\vu_n)\|_{s,2}^2\,\ds\bigg]^\frac{p}{2}\\
&\lesssim c(R,T) \left( \kappa \E \left[ \sup_{t\in(0,T)}\| (r_n, \vu_n) \|_{s,2}^{2p} \right] +c(\kappa)\E\bigg[\int_0^T\|(r_n ,\vu_n)\|_{s,2}^2\,\ds\bigg]^p \right)
\end{align*}
where we also took into account the Burgholder-Davis-Gundy and weighted Young inequalities.
Finally, we apply the Gronwall lemma to conclude
\begin{align} \label{E14''}
&\E\bigg[ \left( \sup_{(0,T)}\left\| (r_n,\bfu_n) \right\|^2_{s,2}  +  \int_0^T\int_{\mt}|\nabla^{s+1}\vu_n|^2 \dx \dt \right)^p \bigg]
\lesssim c(R,T, s) \E\bigg[\left\| (r_0,\bfu_0) \right\|^{2p}_{s,2} +1\bigg]
\end{align}
whenever $s > \frac{N}{2} + 2$.

\subsection{Compactness}
\label{subsec:comp}

Now we have all in hand to set up our compactness argument leading to the existence part of Theorem \ref{thm:appr}.
Let us define the path space $\mathcal{X}=\mathcal{X}_r\times\mathcal{X}_\bu\times\mathcal X_W$,
\begin{align*}
\mathcal{X}_\bu&=  C([0,T]; W^{\beta,2}(\mt; \mr^N)) ,\quad
\mathcal{X}_r= C([0,T]; W^{\beta,2}(\mt)) ,\quad\mathcal{X}_W=C([0,T];\mathfrak{U}_0),
\end{align*}
\color{black}
where $\beta < s$ (not necessarily integer) can be chosen arbitrarily close to $s$, in particular, $\beta > \frac{N}{2} + 2$ so that we have the embedding
\[
W^{\beta,2}(\mt)\hookrightarrow W^{2,\infty}(\mt)
\]
needed to pass to the limit in the cut-off operators.

We denote by $\mu_{r_n}$ and $\mu_{\bu_n}$ the law of $r_n$ and $\bu_n$
on the corresponding path space. By $\mu_W$ we denote the law of $W$ on $\mathcal{X}_W$ and their joint law on $\mathcal{X}$ is $\mu^n$.
To proceed, it is necessary to establish tightness of $\{\mu^n;\,n\in\mathbb N\}$.

\begin{Proposition}\label{prop:bfutightness}
The set $\{\mu_{\bu_n};\,n\in\mathbb N\}$ is tight on $\mathcal{X}_\bu$.
\end{Proposition}
\begin{proof}
We start with a compact embedding relation
\[
C([0,T]; W^{s,2}(\mt)) \cap C^\gamma ([0,T]; L^2(\mt)) \hookrightarrow\hookrightarrow C([0,T]; W^{\beta,2}(\mt)),\ \gamma > 0, \ \beta < s,
\]
that follows directly from the abstract Arzel\` a-Ascoli theorem.

Due to \eqref{est3}, $\bfu_n$ satisfies
\begin{align*}
\begin{aligned}
\vu_n(t)&=P_n\vu_0 - \int_0^t\varphi_R(\|\vu_n\|_{2,\infty})P_n\big[ \vu_n \cdot \Grad \vu_n + r_n \Grad r_n \big]  \ds \\ &+\int_0^t \varphi_R(\|\vu_n\|_{2,\infty}) P_n\big[D(r_n) \Div \mathbb{S} (\Grad \vu_n)\big] \ds
 +\int_0^t  \varphi_R(\|\vu_n\|_{2,\infty})  P_n\mathbb{F} (r_n, \vu_n) \D W.
\end{aligned}
\end{align*}
Now we decompose $\bfu_n$ into two parts, namely, $\bfu_n=\bfY_n+\bfZ_n$, where
\begin{equation*}
 \begin{split}
\bfY_n(t)&=P_n\vu_0 - \int_0^t\varphi_R(\|\vu_n\|_{2,\infty})P_n\big[ \vu_n \cdot \Grad \vu_n + r_n \Grad r_n \big]  \ds \\
&+\int_0^t \varphi_R(\|\vu_n\|_{2,\infty})P_n\big[D(r_n) \Div \mathbb{S} (\Grad \vu_n)\big] \ds,\\
\bfZ_n(t)&=\int_0^t  \varphi_R(\|\vu_n\|_{2,\infty})P_n\tn{F} (r_n, \vu_n) \D W.
 \end{split}
\end{equation*}
By \eqref{E14''} and the continuity of $P_n$ on $L^{2}$ we have for any $\kappa\in (0,1)$ that
\begin{equation*}
\stred \left[ \|\bfY_n\|_{C^\kappa([0,T];L^2(\mt))} \right]\leq c(R),
\end{equation*}
while \eqref{E14''} combined with Corollary \ref{Cflan} (for $s = 0$), Remark \ref{Rflan} yields
the same conclusion for $\vc{Z}_n$, with $0 < \kappa < 1/2$.
\end{proof}

\begin{Proposition}\label{proprtightness}
The set $\{\mu_{r_n};\,n\in\mathbb N\}$ is tight on $\mathcal{X}_r$.
\end{Proposition}
\begin{proof}
The proof is completely analogous to Proposition \ref{prop:bfutightness}
using the equation \eqref{E3'} for $r_n$ and the uniform estimate \eqref{E14''}.
\end{proof}

Since also the law $\mu_W$ is tight as being a Radon measure on the Polish space $\mathcal{X}_W$ we can finally deduce tightness of the joint laws $\mu^n$.

\begin{Corollary}\label{cor:tight}
The set $\{\mu^n;\,n\in\mathbb N\}$ is tight on $\mathcal{X}$.
\end{Corollary}

Since the path space $\mathcal{X}$ is a Polish space we may use the classical Skorokhod representation theorem. That is, passing to a weakly convergent subsequence $\mu^\varepsilon$ (and denoting by $\mu$ the limit law) we infer the following result.

%
%
%

\begin{Proposition}\label{prop:skorokhod1}
There exists a subsequence $\mu^n$, a probability space $(\tilde\Omega,\tilde{\mathfrak F},\tilde\prst)$ with $\mathcal{X}$-valued Borel measurable random variables $(\tilde r_n,\tilde\bu_n,\tilde W_n)$, $N\in\mn$, and $(\tilde r,\tilde\bu,\tilde W)$ such that
\begin{enumerate}
 \item the law of $(\tilde r_n,\tilde\bu_n,\tilde W_n)$ is given by $\mu^n$, $n\in\mathbb N$,
\item the law of $(\tilde r,\tilde\bu,\tilde W)$ is given by $\mu$,
 \item $(\tilde r_n,\tilde\bu_n,\tilde W_n)$ converges $\,\tilde{\prst}$-a.s. to $(\tilde r,\tilde{\bu},\tilde{W})$ in the topology of $\mathcal{X}$.
\end{enumerate}
\end{Proposition}

\subsection{Identification of the limit}
\label{subsec:ident}

As the next step, we will identify the limit obtained in Proposition \ref{prop:skorokhod1} with a strong martingale solution to \eqref{E3'}--\eqref{E4'}, completing the proof of Theorem \ref{thm:appr}.

Let us first fix some notation that will be used in the sequel. We denote by $\bfr_t$ the operator of restriction to the interval $[0,t]$ acting on various path spaces. In particular, if $X$ stands for one of the path spaces $\mathcal{X}_r,\,\mathcal{X}_{\bfu}$ or $\mathcal{X}_{W}$ and $t\in[0,T]$, we define
\begin{align*}
\bfr_t:X\rightarrow X|_{[0,t]},\quad f\mapsto f|_{[0,t]}.
\end{align*}
Clearly, $ \bfr_t$ is a continuous mapping.
Let $(\tilde{\mathfrak F}_t^n)$ and $(\tilde{\mathfrak F}_t)$, respectively, be the $\tilde{\prst}$-augmented canonical filtration of the process $(\tilde r_n,\tilde{\bu}_n,\tilde{W}_n)$ and $(\tilde r,\tilde{\bu},\tilde{W})$, respectively, that is
\begin{equation*}
\begin{split}
\tilde{\mathfrak F}_t^n&=\sigma\big(\sigma\big(\bfr_t\tilde r_n,\,\bfr_t\tilde{\bu}_n,\,\bfr_t \tilde{W}_n\big)\cup\big\{\mathscr M\in\tilde{\mathfrak F};\;\tilde{\prst}(\mathscr M)=0\big\}\big),\quad t\in[0,T],\\
\tilde{\mathfrak F}_t&=\sigma\big(\sigma\big(\bfr_t\tilde{r},\,\bfr_t\tilde{\bu},\,\bfr_t\tilde{W}\big)\cup\big\{\mathscr M\in\tilde{\mathfrak F};\;\tilde{\prst}(\mathscr M)=0\big\}\big),\quad t\in[0,T].
\end{split}
\end{equation*}
We claim that $(\tilde r,\tilde\vu,\tilde W)$ is a strong martingale solution to \eqref{E3'}-\eqref{E4'}. Indeed, in order to identify \eqref{E3'}, let us define the following functional
\begin{align*}
(r,\vu)&\mapsto L(r,\bfu)_t:={r}(t)-{r}(0) + \int_0^t\varphi_R(\|{\vu}\|_{2,\infty})\Big[{\vu} \cdot \Grad {r} \   -\tfrac{\gamma - 1}{2}\,{r}\, \Div {\vu}\Big]\ds.
\end{align*}
Since the couple $(r_n,\vu_n)$ solves \eqref{E3'} on the original probability space, it holds $L(r_n,\vu_n)_t=0$, $t\in[0,T]$. Thus, due to equality of laws we get
$$\tilde\E\|L(\tilde r_n,\tilde\vu_n)_t\|_{2}^2=\E\|L(r_n,\vu_n)_t\|_2^2=0.$$
With Proposition \ref{prop:skorokhod1} and \eqref{E14''} at hand, we may pass to the limit on the left hand side and deduce that $(\tilde r,\tilde\vu)$ solves \eqref{E3'}.

In order to identify \eqref{E4'}, we first note that since $\tilde W_n$ has the same law as $W$, there exists a collection of mutually independent real-valued $(\tilde{\mathfrak F}_t)$-Wiener processes $(\tilde{\beta}^{n}_k)_{k\geq1}$ such that $\tilde{W}_n=\sum_{k\geq1}\tilde{\beta}^{n}_k e_k$ , i.e. there exists a collection of mutually independent real-valued $(\tilde{\mathfrak F}_t)$-Wiener processes $(\tilde{\beta}_k)_{k\geq1}$ such that $\tilde{W}=\sum_{k\geq1}\tilde{\beta}_k e_k$.
As the next step, let us fix times $s,t\in[0,T]$ such that $s<t$ and let
$$h:\mathcal{X}_r|_{[0,s]}\times\mathcal{X}_\bfu|_{[0,s]}\times\mathcal{X}_W|_{[0,s]}\rightarrow [0,1]$$
be a continuous function.
We define functionals
\begin{align*}
(r,\vu)&\mapsto M^n(r,\bfu)_t:={\vu}(t)-{\vu}(0) + \int_0^t\varphi_R(\|{\vu}\|_{2,\infty})P_n\big[ {\vu}\cdot \Grad {\vu} +{r} \Grad {r} \big]  \ds\\
&\qquad\qquad\qquad\qquad  -\int_0^t \varphi_R(\|{\vu}\|_{2,\infty})P_n\big[D({r}) \Div \mathbb{S} (\Grad {\vu})\big] \ds\\
(r,\vu)&\mapsto M(r,\bfu)_t:={\vu}(t)-{\vu}(0) + \int_0^t\varphi_R(\|{\vu}\|_{2,\infty})\big[ {\vu}\cdot \Grad {\vu} +{r}\Grad {r} \big]  \ds\\
&\qquad\qquad\qquad\qquad  -\int_0^t \varphi_R(\|{\vu}\|_{2,\infty})\big[D({r}) \Div \mathbb{S} (\Grad {\vu})\big] \ds.
\end{align*}
Since $(r_n,\vu_n)$ satisfies \eqref{E4'} on the original probability space, we have that
$$M^n(r_n,u_n)_t=\int_0^t\varphi_R(\|\vu_n\|_{2,\infty})P_n\mathbb F(r_n,u_n)\,\dif W.$$
Hence $M^n(r_n,\vu_n)$ is an $L^2(\mt)$-valued martingale and if $(f_j)$ is an orthonormal basis of $L^2(\mt)$ then for all $j\in\mn$
$$\E\left[h(\mathbf{r}_s r_n,\mathbf{r}_s \vu_n,\mathbf{r}_s W)\langle M^n(r_n,\vu_n)_t-M^n(r_n,\vu_n)_s,f_j\rangle\right]=0,$$
\begin{align*}
\E\bigg[h(\mathbf{r}_s r_n,\mathbf{r}_s \vu_n,\mathbf{r}_s W)&\bigg(\langle M^n(r_n,\vu_n)_t,f_j\rangle^2-\langle M^n(r_n,\vu_n)_s,f_j\rangle^2\\
&\qquad-\int_s^t\varphi_R(\|\vu_n\|_{2,\infty})\|(P_n\mathbb{F}(r_n,\vu_n))^*f_j\|_\mathfrak{U}^2\,\dif \sigma\bigg)\bigg]=0,
\end{align*}
\begin{align*}
\E\bigg[h(\mathbf{r}_s r_n,\mathbf{r}_s \vu_n,\mathbf{r}_s W)&\bigg(\beta_k(t)\langle M^n(r_n,\vu_n)_t,f_j\rangle-\beta_k(s)\langle M^n(r_n,\vu_n)_s,f_j\rangle\\
&\qquad-\int_s^t\varphi_R(\|\vu_n\|_{2,\infty})\langle e_k,(P_n\mathbb{F}(r_n,\vu_n))^*f_j\rangle_\mathfrak{U}\,\dif \sigma\bigg]=0.
\end{align*}
Equality of laws implies the corresponding three expressions for $(\tilde r_n,\tilde \vu_n)$ and finally due to Proposition \ref{prop:skorokhod1} and the uniform moment estimates from \eqref{E14''} we may pass to the limit to deduce
$$\tilde\E\left[h(\mathbf{r}_s \tilde r,\mathbf{r}_s \tilde\vu,\mathbf{r}_s \tilde W)\langle M(\tilde r,\tilde\vu)_t-M^n(\tilde r,\tilde\vu)_s,f_j\rangle\right]=0,$$
\begin{align*}
\tilde\E\bigg[h(\mathbf{r}_s \tilde r,\mathbf{r}_s \tilde\vu,\mathbf{r}_s\tilde W)&\bigg(\langle M(\tilde r,\tilde\vu)_t,f_j\rangle^2-\langle M(\tilde r,\tilde\vu)_s,f_j\rangle^2\\
&\qquad-\int_s^t\varphi_R(\|\tilde u\|_{2,\infty})\|(\mathbb{F}(\tilde r,\tilde\vu))^*f_j\|_\mathfrak{U}^2\,\dif \sigma\bigg)\bigg]=0,
\end{align*}
\begin{align*}
\tilde\E\bigg[h(\mathbf{r}_s \tilde r,\mathbf{r}_s \tilde \vu,\mathbf{r}_s \tilde W)&\bigg(\tilde\beta_k(t)\langle M(\tilde r,\tilde\vu)_t,f_j\rangle-\tilde\beta_k(s)\langle M(\tilde r,\tilde\vu)_s,f_j\rangle\\
&\qquad-\int_s^t\varphi_R(\|\tilde u\|_{2,\infty})\langle e_k,(\mathbb{F}(\tilde r,\tilde\vu))^*f_j\rangle_\mathfrak{U}\,\dif \sigma\bigg]=0.
\end{align*}
According to \cite[Proposition A.1]{degen} this finally yields \eqref{E4'}
and completes { the existence part of the proof of Theorem \ref{thm:appr}. Note that the strong continuity of
$r$ and $\vu$ in $W^{s,2}(\mt)$ $\prst\text{-a.s.}$ can be deduced directly from the equations.} Indeed, using the variational approach, the momentum equation \eqref{E4'} is solved in the Gelfand triplet
$$W^{s+1,2}(\mt;\mr^N)\hookrightarrow W^{s,2}(\mt;\mr^N)\hookrightarrow W^{s-1,2}(\mt;\mr^N),$$
the stochastic integral has continuous trajectories in $W^{s,2}(\mt;\mr^N)$ due to the uniform estimates, Corollary \ref{Cflan} (part (ii)) and Remark \ref{Rflan}, while 
the coefficients of the deterministic parts in the momentum equation belong to the space $L^2(0,T; W^{s-1,2}(\mt; \mr^N))$. Hence \cite[Theorem 3.1]{KrRo} applies and yields the desired continuity of the velocity field $\vu$.
The continuity of $r$ then follows from the equation of continuity.

\subsection{Pathwise uniqueness}
\label{subsec:uniq}

To show pathwise uniqueness, we mimick the approach of Subsection \ref{UNIF}. The difference of two solutions
$(r^j, \vu^j)$, $j =1,2$, satisfies
\begin{equation}\label{UU1}
\begin{split}
\D \partial^\alpha_x (r^1 - r^2)  &= - \varphi_R \left( \| \vu^1 \|_{W^{2, \infty}} \right) \partial^\alpha_x \left( \vu^1 \cdot \Grad r^1  + \frac{\gamma-1}{2} r^1 \Div \vc{u}^1 \right) \dt
\\ &+ \varphi_R \left( \| \vu^2 \|_{W^{2, \infty}} \right) \partial^\alpha_x \left( \vu^2 \cdot \Grad r^2  + \frac{\gamma-1}{2} r^2 \Div \vc{u}^2 \right) \dt,
\end{split}
\end{equation}
and
\[
\begin{split}
\D  \partial^\alpha_x (\vc{u}_1 - \vc{u}_2 )   &= - \varphi_R \left( \| \vu_1 \|_{W^{2, \infty}} \right) \partial^\alpha_x \Big( \vu_1 \cdot \Grad \vu_1   + r_1 \Grad r_1 - D(r_1) \Div \mathbb{S}(\Grad \vu_1) \Big) \dt \\
&+\varphi_R \left( \| \vu_2 \|_{W^{2, \infty}} \right) \partial^\alpha_x \Big( \vu_2 \cdot \Grad \vu_2   + r_2 \Grad r_2 - D(r_2) \Div \mathbb{S}(\Grad \vu_2) \Big) \dt\\
&+  \Big[ \varphi_R \left( \| \vu_1 \|_{W^{2, \infty}} \right) \partial^\alpha_x \mathbb{F}(r_1, \vc{u}_1) - \varphi_R \left( \| \vu_2 \|_{W^{2, \infty}} \right) \partial^\alpha_x \mathbb{F}(r_2, \vc{u}_2)
\Big]  {\rm d}W
\end{split}
\]
for $|\alpha| \leq m$.

Multiplying (\ref{UU1}) on $\partial^\alpha_x (r^1 - r^2)$, we get
\begin{equation}\label{UU2}
\begin{split}
\frac{1}{2} \D \left| \partial^\alpha_x (r^1 - r^2) \right|^2  &= - \varphi_R \left( \| \vu^1 \|_{W^{2, \infty}} \right) \partial^\alpha_x \left( \vu^1 \cdot \Grad r^1  + \frac{\gamma-1}{2} r^1 \Div \vc{u}^1 \right) \partial^\alpha_x (r^1 - r^2)\dt
\\ &+ \varphi_R \left( \| \vu^2 \|_{W^{2, \infty}} \right) \partial^\alpha_x \left( \vu^2 \cdot \Grad r^2  + \frac{\gamma-1}{2} r^2 \Div \vc{u}^2 \right) \partial^\alpha_x (r^1 - r^2)\dt.
\end{split}
\end{equation}
Similarly, using It\^{o}'s product rule we obtain
\begin{equation} \label{UU3}
\begin{split}
\frac{1}{2} & \D  \left| \partial^\alpha_x (\vc{u}^1 - \vc{u}^2 ) \right|^2 \\  &= - \varphi_R \left( \| \vu^1 \|_{W^{2, \infty}} \right) \partial^\alpha_x \Big( \vu^1 \cdot \Grad \vu^1   + r^1 \Grad r^1 - D(r^1) \Div \mathbb{S}(\Grad \vu^1) \Big) \cdot \partial^\alpha_x (\vc{u}^1 - \vc{u}^2 ) \dt \\
&+\varphi_R \left( \| \vu^2 \|_{W^{2, \infty}} \right) \partial^\alpha_x \Big( \vu^2 \cdot \Grad \vu^2   + r^2 \Grad r^2 - D(r^2) \Div \mathbb{S}(\Grad \vu^2) \Big) \cdot \partial^\alpha_x (\vc{u}^1 - \vc{u}^2 )\dt\\
&+  \Big[ \varphi_R \left( \| \vu^1 \|_{W^{2, \infty}} \right) \partial^\alpha_x \mathbb{F}(r^1, \vc{u}^1) - \varphi_R \left( \| \vu^2 \|_{W^{2, \infty}} \right) \partial^\alpha_x \mathbb{F}(r^2, \vc{u}^2)
\Big] \cdot \partial^\alpha_x (\vc{u}^1 - \vc{u}^2 )  {\rm d}W\\
&+ \frac{1}{2} \Big( \varphi_R \left( \| \vu^1 \|_{W^{2, \infty}} \right) \partial^\alpha_x \mathbb{F}(r^1, \vc{u}^1) - \varphi_R \left( \| \vu^2 \|_{W^{2, \infty}} \right) \partial^\alpha_x \mathbb{F}(r^2, \vc{u}^2)
\Big)^2 {\rm d}t
\end{split}
\end{equation}
Now observe, by virtue of the standard embedding relation,
\[
\left| \varphi_R \left( \| \vu^1 \|_{W^{2, \infty}} \right)-  \varphi_R \left( \| \vu^2 \|_{W^{2, \infty}} \right)
\right| \leq c_1(R) \left\| \vu^1 - \vu^2 \right\|_{W^{2, \infty}} \leq c_2 (R) \left\| \vu^1 - \vu^2 \right\|_{W^{m, 2}}
\]
as soon as $m > \frac{N}{2} + 2$. Thus we sum \eqref{UU2}, \eqref{UU3}, integrate over the physical space, and perform the same estimates as in Section \ref{UNIF} noting that the highest order terms in \eqref{UU2} read
\[
\begin{split}
\varphi_R \left( \| \vu^1 \|_{W^{2, \infty}} \right) &\int_{ \tor}\left( \vu^1 \cdot \Grad \partial^\alpha_x r^1 - \vu^2 \cdot \Grad \partial^\alpha_x r^2 \right)
\partial^\alpha_x \left( r^1 - r^2 \right) \ \dx\\
+ \frac{\gamma - 1}{2} \varphi_R \left( \| \vu^1 \|_{W^{2, \infty}} \right) &\int_{ \tor} \left( r^1 \Div \partial^\alpha_{x} \vu^1 - r^2 \Div \partial^\alpha_{x} \vu^2 \right) \partial^\alpha_x \left( r^1 - r^2 \right) \dx\\
= \varphi_R \left( \| \vu^1 \|_{W^{2, \infty}} \right) &\int_{ \tor}\left( (\vu^1 - \vu^2) \cdot \Grad \partial^\alpha_x r^1
\right)\partial^\alpha_x \left( r^1 - r^2 \right)
 + \frac{1}{2} \Div \vu^2
\left| \partial^\alpha_x \left( r^1 - r^2 \right) \right|^2 \ \dx\\
+ \frac{\gamma - 1}{2} \varphi_R \left( \| \vu^1 \|_{W^{2, \infty}} \right) &\int_{ \tor}  (r^1  - r^2) \Div \partial^\alpha_{x} \vu^2  \partial^\alpha_x \left( r^1 - r^2 \right) \dx \\
+ \frac{\gamma - 1}{2} \varphi_R \left( \| \vu^1 \|_{W^{2, \infty}} \right) &\int_{ \tor} r^1 \Div \partial^\alpha_{x} (\vu^1 - \vu^2)
\partial^\alpha_x \left( r^1 - r^2 \right) \dx
\end{split}
\]
where the last integral
\[
\varphi_R \left( \| \vu^1 \|_{W^{2, \infty}} \right) \int_{\mt} r^1 \Div \left( \partial^\alpha_x (\vu^1 - \vu^2) \right) \partial^\alpha_x
(r^1 - r^2)\ \dx
\]
cancels, after by parts integration, with its counterpart in \eqref{UU3}, namely
\[
\varphi_R \left( \| \vu^1 \|_{W^{2, \infty}} \right) \int_{\mt} r^1 \left( \partial^\alpha_x (\vu^1 - \vu^2) \right) \cdot \Grad \partial^\alpha_x
(r^1 - r^2)\ \dx.
\]

Thus we deduce, exactly as in Subsection \ref{UNIF},
\begin{equation*}
\begin{split}
\D &\left( \left\| r^1 - r^2 \right\|^2_{W^{m,2}} + \left\| \vu^1 - \vu^2 \right\|^2_{W^{m,2}} \right)
\\
& \leq c(R)  \left[ \left( 1 + \sum_{j=1}^2 \left( \| r^j \|_{W^{m+1,2}}^2 + \| \vu^j \|_{W^{m+2,2}}^2 \right)\right)
\left( \left\| r^1 - r^2 \right\|^2_{W^{m,2}} + \left\| \vu^1 - \vu^2 \right\|^2_{W^{m,2}} \right) \right] \dt\\
& +   \Big[ \varphi_R \left( \| \vu^1 \|_{W^{2, \infty}} \right) \partial^\alpha_x \mathbb{F}(r^1, \vc{u}^1) - \varphi_R \left( \| \vu^2 \|_{W^{2, \infty}} \right) \partial^\alpha_x \mathbb{F}(r^2, \vc{u}^2)
\Big] \cdot \partial^\alpha_x (\vc{u}^1 - \vc{u}^2 )  {\rm d}W,
\end{split}
\end{equation*}
where $m > \frac{N}{2} + 2$.
Let us now set
$$G(t)=c(R) \left( 1 + \sum_{j=1}^2 \left( \| r^j(t) \|_{W^{m+1,2}}^2 + \| \vu^j (t)\|_{W^{m+2,2}}^2 \right)\right)$$
and observe that if $s\geq m+1$ then the a priori estimates from Subsection \ref{UNIF} imply in particular that $G\in L^1(0,T)$ a.s. Applying the It\^o formula to the product we therefore obtain
\begin{align*}
&\D\left[\mathrm{e}^{-\int_0^t G(\sigma)\D\sigma}\Big(\|r^1-r^2\|_{m,2}^2+\|\vu^1-\vu^2\|_{m,2}^2\Big)\right]\\
&=-G(t)\mathrm{e}^{-\int_0^t G(\sigma)\D\sigma}\Big(\|r^1-r^2\|_{m,2}^2+\|\vu^1-\vu^2\|_{m,2}^2\Big)\dt\\
&\qquad+\mathrm{e}^{-\int_0^t G(\sigma)\D\sigma}\D\Big(\|r^1-r^2\|_{m,2}^2+\|\vu^1-\vu^2\|_{m,2}^2\Big)\\
&\leq   \mathrm{e}^{-\int_0^t G(\sigma)\D\sigma}\Big[ \varphi_R \left( \| \vu^1 \|_{W^{2, \infty}} \right) \partial^\alpha_x \mathbb{F}(r^1, \vc{u}^1) - \varphi_R \left( \| \vu^2 \|_{W^{2, \infty}} \right) \partial^\alpha_x \mathbb{F}(r^2, \vc{u}^2)
\Big] \cdot \partial^\alpha_x (\vc{u}^1 - \vc{u}^2 )  {\rm d}W(t).
\end{align*}
Integrating over $[0,t]$ and taking expectation we observe that the stochastic integral vanishes due to the assumptions on $r,\vu$ in Definition \ref{def:strsolmart} and consequently we may infer that
\[
\E \left[\mathrm{e}^{-\int_0^t G(\sigma)\D\sigma}\Big( \left\| r^1(t) - r^2 (t)\right\|^2_{W^{m,2}} + \left\| \vu^1(t) - \vu^2 (t)\right\|^2_{W^{m,2}} \Big)\right]  = 0
\]
whenever
\[
\E \left[ \left\| r^1_0 - r^2_0 \right\|^2_{W^{m,2}} + \left\| \vu^1_0 - \vu^2_0 \right\|^2_{W^{m,2}} \right] = 0.
\]
Since
$$\mathrm{e}^{-\int_0^t G(\sigma)\D\sigma}>0\qquad\p\text{-a.s.}$$
and the trajectories of $r^i,\vu^i$, $i=1,2,$ are continuous in $W^{m,2}(\mt)$,
the pathwise
uniqueness  from Theorem \ref{thm:appr} follows.

\subsection{\textcolor{black}{Existence of a strong pathwise approximate solution}}
\label{subsec:pathwiseexist}

In order to complete the proof of Theorem \ref{thm:appr}, we make use of
the Gy\"{o}ngy--Krylov characterization of convergence in probability introduced in \cite[Lemma 1.1]{krylov}.
It applies to situations when pathwise uniqueness and existence of a martingale solution are valid and allows to establish existence of a pathwise solution.
%
\begin{Lemma} \label{GK} 
Let $(\mathcal{X},\tau)$ be a Polish space and $\left\{ Y_n; n \in \mathbb{}{N} \right\}$ a family of random variables ranging in $\mathcal{X}$. Let
\[
\nu_{m,n} \equiv \prst \left[ [Y_m, Y_n ] \in B \right], \ B \ \mbox{a Borel set in}\ \mathcal{X} \times \mathcal{X}
\]
be the collection of joint laws.
Then $Y_n$ converges in probability only if any subsequence of joint probability laws $\{ \nu_{m_k, n_k} \}_{k \geq 0}$
contains a weakly converging subsequence to a $\nu$ such that
\[
\nu \left[ (u,v) \in \mathcal{X} \times \mathcal{X}, u = v \right] = 1.
\]
\end{Lemma}

We start with a regular initial initial data corresponding to $s > \frac{N}{2} + 3$ required for pathwise uniqueness of
strong solutions to the approximate problem \eqref{E3'}, \eqref{E4'}.
Going back to the construction of approximate solution we
denote by $\mu_{m,n}$ the  joint law  of
$$(r_m,\bfu_m,r_n,\bfu_n)\quad\text{on the space}\quad \mathcal{X}_r\times \mathcal{X}_{\bfu}\times \mathcal{X}_r\times \mathcal{X}_{\bfu},$$
where $r_m$, $\vu_n$, $r_n$, $\vu_n$ are the Galerkin solutions.
In addition, denoting $\mu_W$ the law of $W$ on $\mathcal{X}_W$, we introduce
the extended path space
$$\mathcal{X}^J= \mathcal{X}_r\times \mathcal{X}_{\bfu}\times \mathcal{X}_r\times \mathcal{X}_{\bfu}\times\mathcal{X}_W$$
and denote by $\nu_{m,n}$  the joint law of
$$(r_m,\bfu_m,r_n,\bfu_n,W)\quad\text{on}\quad\mathcal{X}^J.$$
%
The following result follows easily from the arguments of Subsection \ref{subsec:comp}.

\begin{Proposition}
The collection $\{\nu_{m,n};\,m,n\in\mn\}$ is tight on $\mathcal{X}^J$.
\end{Proposition}

Let us take any subsequence $\{\nu_{m_k,n_k};\,k\in\mn\}$. By the Skorokhod representation theorem, we infer (for a further subsequence but without loss of generality we keep the same notation) the existence of a probability space $(\bar{\Omega},\bar{\mf},\bar{\prst})$ with a sequence of random variables
$$(\hat r_{n_k},\hat\bfu_{n_k},\check r_{m_k},\check\bfu_{m_k},\bar W_k),\quad k\in\mn,$$
conver\-ging almost surely in $\mathcal{X}^J$ to a random variable
$$(\hat r,\hat\bfu,\check r,\check\bfu,\bar W)$$
and
$$\bar{\prst}\big((\hat r_{n_k},\hat\bfu_{n_k},\check r_{m_k},\check\bfu_{m_k},\bar W_k)\in\,\,\cdotp\big)=\nu^{n_k,m_k}(\cdot).$$
Observe that in particular, $\mu^{n_k,m_k}$ converges weakly to a measure $\mu$ defined by
$$\mu(\cdot)=\bar{\prst}\big((\hat r,\hat\bfu,\check r,\check\bfu)\in\,\,\cdotp\big).$$
As the next step, we recall the technique established in Subsection \ref{subsec:ident}. Analogously, it can be applied to both
$$(\hat r_{n_k},\hat\bfu_{n_k},\bar W_k),\;(\hat r,\hat\bfu,\bar W)$$
and
$$(\check r_{m_k},\check\bfu_{m_k},\bar W_k),\;(\check r,\check\bfu,\bar W)$$
in order to show that $(\hat r, \hat\bfu,\bar W)$ and $(\check r, \check\bfu,\bar W)$ are strong martingale solutions to the approximate system \eqref{E3'}--\eqref{E4'}.
Finally, since $r_{n_k}(0)=r_{m_k}(0)=r_{0}$, it follows that
\begin{align*}
\bar\p(\hat r(0)=\check r(0))=1.
\end{align*}
Since   $\vu_{n_k}(0)=P_{n_k}\vu_0$, $\vu_{m_k}(0)=P_{m_k}\vu_0$, we obtain for every $\ell\leq n_k\wedge m_k$
\begin{align*}
\bar\p(P_{\ell}\hat \vu_{n_k}(0)=P_\ell\check \vu_{m_k}(0))=\p(P_{\ell} \vu_{n_k}(0)=P_\ell\vu_{m_k}(0))=1
\end{align*}
which leads to
$$\bar\p(\hat \vu(0)=\check\vu(0))=1.$$
Hence, in accordance with the pathwise uniqueness established in Theorem \ref{thm:appr}, we get the desired conclusion
\begin{align*}
\mu&\Big((r_1,\bfu_1,r_2,\bfu_2);\;(r_1,\bfu_1)=(r_2,\bfu_2)\Big)=\bar{\prst}\Big((\hat r,\hat\bfu)=(\check r,\check\bfu)\Big)=1.
\end{align*}

Thus, we have all in hand to apply Lemma \ref{GK}, which implies that the original sequence $(r_n,\bfu_n)$ defined on the initial probability space $(\Omega,\mathfrak F,\prst)$ converges in probability in the topology of $\mathcal{X}_r\times \mathcal{X}_{\bfu}$ to a random variable $(r,\bfu)$. Without loss of generality, we assume that the convergence is almost sure and again by the method from Subsection \ref{subsec:ident} we finally deduce that
the limit
is the unique strong pathwise solution to the approximate problem  \eqref{E3'}--\eqref{E4'}. Let us denote this solution by $(r_R, \vu_R)$.

\section{\textcolor{black}{Proof of the main result, Theorem \ref{thm:main}}}
\label{subsec:strong}

Throughout the remainder of the paper, we go back to the original problem \eqref{P1}--\eqref{P4} and prove Theorem \ref{thm:main}. Our approach relies  on the equivalence between \eqref{P1}--\eqref{P2} and \eqref{E3}--\eqref{E4} which is valid provided the density remains strictly positive, cf. Subsection \ref{subsec:symsyst}. In addition, introducing suitable stopping times allows us to work with  \eqref{E3'}--\eqref{E4'} instead of \eqref{E3}--\eqref{E4} and therefore we may apply the results of the previous section, namely, Theorem \ref{thm:appr}.
Nevertheless, there is an  additional difficulty that originates in  the fact that the initial condition is not assumed to be integrable in $\omega$ and the initial density is not bounded from below by a positive constant. Consequently, the a priori estimates from Subsection \ref{UNIF} are no longer valid and the initial condition has to be truncated for Theorem \ref{thm:appr} to be applicable.
For this reason, the proof of uniqueness as well as existence of a local strong pathwise solution is divided into two steps. First, we consider an additional assumption the initial data so that Theorem \ref{thm:appr} applies. Second, we avoid  this hypothesis.

\subsection{Uniqueness}
\label{s:un}

Let us first take an additional assumption that
\begin{equation} \label{DATA1}
\vr_0 \in  L^\infty(\Omega; \mathfrak{F}_0, \prst, W^{s,2}(\mt) ),\   \vu_0 \in L^\infty(\Omega; \mathfrak{F}_0, \prst, W^{s,2}(\mt,\mr^N) ),
\ \vr_0 > \underline{\vr} > 0
\end{equation}
for some deterministic constant $\underline{\vr} > 0$.
In this case, the pathwise uniqueness of \eqref{P1}--\eqref{P4} is a simple consequence of the pathwise uniqueness for \eqref{E3'}--\eqref{E4'} proved in Theorem \ref{thm:appr}.
To be more precise, let  $[\varrho^i,\vu^i,(\mathfrak{t}_R^i),\mathfrak{t}^i]$, $i=1,2,$ be two maximal strong pathwise solutions to \eqref{P1}--\eqref{P4} starting from $[\varrho_0,\vu_0]$ satisfying \eqref{DATA1}.
Then
$$\left[r^i:=\sqrt{\frac{2a\gamma}{\gamma-1}}{(\varrho^i)}^{\frac{\gamma-1}{2}},\vu^i\right],\quad i=1,2,$$
both solve \eqref{E3'}--\eqref{E4'} up to the stopping time $\mathfrak{t}_R^1\wedge\mathfrak{t}_R^2$ and their initial conditions coincide. Besides, the a priori estimates from Subsection \ref{UNIF} as well as the pathwise uniqueness from Subsection \ref{subsec:uniq} apply up to the stopping time $\mathfrak{t}_R^1\wedge\mathfrak{t}_R^2$
and we deduce that
$$
\p\Big([\varrho^1,\vu^1](t\wedge\mathfrak{t}^1_R\wedge\mathfrak{t}^2_R)=[\varrho^2,\vu^2](t\wedge\mathfrak{t}^1_R\wedge\mathfrak{t}^2_R),\ \text{for all }t\in[0,T]\Big)=1.
$$
Sending $R\to\infty$ implies by dominated convergence
$$
\p\Big([\varrho^1,\vu^1](t\wedge\mathfrak{t}^1\wedge\mathfrak{t}^2)=	[\varrho^2,\vu^2](t\wedge\mathfrak{t}^1\wedge\mathfrak{t}^2),\ \text{for all }t\in[0,T]\Big)=1.
$$
As a consequence, the two solutions coincide up to the stopping time $\mathfrak{t}^1\wedge\mathfrak{t}^2$ and due to maximality of  $\mathfrak{t}^1$ as well as $\mathfrak{t}^2$, it necessarily follows that $\mathfrak{t}^1=\mathfrak{t}^2$ a.s. This completes the proof of uniqueness under the additional assumption \eqref{DATA1}.

Let $(\varrho_0,\vu_0)$ satisfy the hypotheses of Theorem \ref{thm:main}, define the set
$$ \Omega_{K}=\left\{\omega\in\Omega \ \Big|\ \|\vu_0(\omega)\|_{s,2}<K, \ \|r_0(\omega)\|_{s,2}<K,\ \inf_{\mt} r_0(\omega)>\frac{1}{K}\right\}$$
and note that $\Omega=\cup_{K\in\mr}\Omega_{K}$.
Therefore, since $\Omega_K$ is $\mathfrak{F}_0$-measurable, the a priori estimates from Subsection \ref{UNIF} can be employed on $\Omega_K$ to obtain
\begin{align} \label{E14''}
&\E\bigg[ \ind_{\Omega_K}\left( \sup_{t\in[0,T\wedge\mathfrak{t}^i_R]}\left\| (r^i(t),\bfu^i(t)) \right\|^2_{s,2}  +  \int_0^{T\wedge\mathfrak{t}^i_R}\|\vu^i(t)\|_{s+1,2}^2 \dt \right)^p \bigg]
\lesssim c(R,T, s,K).
\end{align}
Accordingly, the method of pathwise uniqueness from Subsection \ref{subsec:uniq} can be applied on $\Omega_K$ which yields
$$
\p\Big(\ind_{\Omega_K}[\varrho^1,\vu^1](t\wedge\mathfrak{t}^1_R\wedge\mathfrak{t}^2_R)=\ind_{\Omega_K}[\varrho^2,\vu^2](t\wedge\mathfrak{t}^1_R\wedge\mathfrak{t}^2_R),\ \text{for all }t\in[0,T]\Big)=1
$$
and since $\ind_{\Omega_K}\to \ind_\Omega$, $\mathfrak{t}^i_R\to\mathfrak{t}^i$, $i=1,2,$ a.s., we may send $R,K\to\infty$ and apply the  dominated convergence theorem to deduce that
$$
\p\Big([\varrho^1,\vu^1](t\wedge\mathfrak{t}^1\wedge\mathfrak{t}^2)=	[\varrho^2,\vu^2](t\wedge\mathfrak{t}^1\wedge\mathfrak{t}^2),\ \text{for all }t\in[0,T]\Big)=1.
$$
The uniqueness part of Theorem \ref{thm:main} is thus complete.

\subsection{\textcolor{black}{Existence of a local strong solution for bounded initial data}}
\label{ex1}

Finally, we have all in hand to go back to our original problem \eqref{P1}--\eqref{P4} and establish the existence of a local strong pathwise solution with an a.s. strictly positive stopping time.
Let us first take the additional assumption \eqref{DATA1}.
Having constructed strong solutions for the approximate problem \eqref{E3'}--\eqref{E4'} in Subsection \ref{subsec:pathwiseexist}, which we denoted by $(r_R, \vu_R)$, we define
\[
\tau_R = \inf \left\{ t \in [0,T]\ \Big| \ \| \vu_R (t) \|_{2,\infty} \geq R \right\}
\]
(with the convention $\inf \emptyset=T$). Since $\vu_R$ has continuous trajectories in $W^{s,2}(\mt,\mr^N)$ which is embedded into $ W^{2,\infty}(\mt,\mr^N)$, $\tau_R$ is a well-defined stopping time. Moreover, due to \eqref{DATA1}, the stopping time $\tau_R$  is a.s. positive
provided $R$ is chosen large enough. Next, we recall that, as stated in Theorem \ref{thm:appr},
$$r_R\geq \underline{r}_R>0\quad\text{for a.e.}\quad(\omega,t,x),$$
for some deterministic constant $\underline{r}_R$. Consequently, the density given by
$${\varrho}:=\left(\frac{\gamma-1}{2a\gamma}\right)^{\frac{1}{\gamma-1}}r^{\frac{2}{\gamma-1}}_R,$$
remains uniformly positive as well. Therefore, the unique solution $(r_R,\vu_R)$ of the approximated system \eqref{E3'}--\eqref{E4'} with the initial condition
$$
\left(r_0:=\sqrt{\frac{2a\gamma}{\gamma-1}}{\varrho_0}^{\frac{\gamma-1}{2}} , \vu_0\right)
$$
generates a local strong
pathwise solution
\[
\left({\varrho}:=\left(\frac{\gamma-1}{2a\gamma}\right)^{\frac{1}{\gamma-1}}r^{\frac{2}{\gamma-1}}_R ,\ \vu_R,\tau_R\right)
\]
of the original problem \eqref{P1}--\eqref{P4} with the initial condition $(\varrho_0,\vu_0)$.

\subsection{\textcolor{black}{Existence of a local strong solution for general initial data}}
\label{subsec:ex2}

In order to relax the additional assumption upon the initial datum \eqref{DATA1}, consider again a solution $(r_R,\vu_R)$ of the approximate problem \eqref{E3'}--\eqref{E4'}; now with a stopping time
\[
\tau_K = \tau^1_K \wedge \tau^2_K \wedge \tau^3_K,
\]
\[
\begin{split}
\tau^1_K &= \inf \left\{ t \in [0,T]\ \Big| \ \| \vu_R (t) \|_{s,2} \geq K \right\} \\
\tau^2_K &= \inf \left\{ t \in [0,T]\ \Big| \ \| r_R (t) \|_{s,2} \geq K \right\} \\
\tau^3_K &= \inf \left\{ t \in [0,T]\ \Big| \ \inf_{\mt} r_R (t)  \leq \frac{1}{K} \right\}
\end{split}
\]
with $K=K(R)\to\infty$ as $R\to\infty$ and
\[
K(R) < R \min \left\{ 1 , \frac{1}{c_{1,\infty}}, \frac{1}{c_{2, \infty}} \right\},
\]
where $c_{1,\infty}$, $c_{2, \infty}$ are the constants in the embedding inequalities
\[
\| r \|_{1,\infty} \leq c_{1,\infty} \| r \|_{s,2}, \ \| \vu \|_{2, \infty} \leq c_{2, \infty} \| \vu \|_{s,2}.
\]
The stopping time $\tau_K$ is chosen in such a way that on $[0,\tau_K)$
$$
\sup_{t\in[0,\tau_K]}\|\vu_R(t)\|_{2,\infty}<R,\quad\sup_{t\in[0,\tau_K]}\|r_R(t)\|_{1,\infty}< R,\quad \inf_{t\in[0,\tau_K]}\inf_{\mt}r_R(t)>\frac{1}{R}\quad\p\text{-a.s.}
$$

Next we observe that Theorem \ref{thm:appr} can be used to construct solutions with the stopping time $\tau_{K}$ for general initial data
as in Theorem \ref{thm:main}. Indeed let $(r_0,\vu_0)$ be an $\mathfrak{F}_0$-measurable random variable taking values in $W^{s,2}(\mt)\times W^{s,2}(\mt,\mr^N)$ such that $r_0 > 0$ $\prst$-a.s.
and define the set
\[
U_{K(R)} = \left\{ [r, \vu] \in W^{s,2}(\mt)\times W^{s,2}(\mt,\mr^N)\ \Big|\ \| r \|_{s,2} < K, \ \| \vu \|_{s,2} < K, \ r > \frac{1}{K} \right\}.
\]
Theorem \ref{thm:appr} then provides a (unique) solution $[{r}_M, u_M]$ to \eqref{E3'}--\eqref{E4'} with $R=M$ and with the initial condition $[{r}_0, \vu_0]\ind_{ [r_0, \vu_0] \in \left\{ U_{K(M)} \setminus \cup_{J=1}^{M-1} U_{K(J)} \right\}}$. It  also solves the original system \eqref{E3}, \eqref{E4} up to the stopping time $\tau_{K(M)}$.
Next, we find that
\begin{equation}\label{eq:limsol}
[r, \vu] = \sum_{M=1}^ \infty [r_M, \vu_M] \ind_{ [r_0, \vu_0] \in \left\{ U_{K(M)} \setminus \cup_{J=1}^{M-1} U_{K(J)} \right\}},
\end{equation}
solves the same problem with the initial data $[r_0,\vu_0]$ up to the a.s. strictly positive stopping time
\[
\tau = \sum_{M=1}^\infty \tau_{K(M)} \ind_{ [r_0, \vu_0] \in \left\{ U_{K(M)} \setminus \cup_{J=1}^{M-1} U_{K(J)} \right\}}.
\]
Note that in particular that $[r,\vu]$ has a.s. continuous trajectories in $W^{s,2}(\mt)\times W^{s,2}(\mt,\mr^N)$ and the velocity also belongs to
$ L^2(0,T;W^{s+1,2}(\mt,\mr^N))$ $\p$-a.s.
Indeed, there exists a disjoint collection of sets $\Omega_M\subset\Omega$, $M\in\mn$, satisfying $\cup_M \Omega_M=\Omega$ such that $[r,\vu](\omega)=[r_M,\vu_M](\omega)$ for a.e. $\omega\in\Omega_M$. And due to Theorem \ref{thm:appr}, the trajectories of $[r_M,\vu_M]$ are a.s. continuous in $W^{s,2}(\mt)\times W^{s,2}(\mt,\mr^N)$. On the other hand, we loose the integrability in $\omega$ as the initial condition is only assumed to be in $W^{s,2}(\mt)\times W^{s,2}(\mt,\mr^N)$ a.s. and no integrability in $\omega$ is assumed. In particular, the estimate \eqref{RE2} is no longer valid for the solution \eqref{eq:limsol}.

To conclude, after the straightforward transformation to the original variables $[\vr,\vu]$, we obtain the existence of a local strong pathwise solution to problem \eqref{P1}--\eqref{P4} with a strictly positive stopping time $\tau$.

\subsection{Existence of a maximal strong solution}
\label{ex3}

In order to extend the solution $(\vr,\vu)$ to a maximal time of existence $\mathfrak{t}$, let $\mathcal T$ denote the set of all possible a.s. strictly positive stopping times corresponding to the solution starting from the initial datum $(\vr_0,\vu_0)$. According to the above proof, this set is nonempty. Moreover, it is closed with respect to finite minimum and finite maximum operations. More precisely,
$$\sigma_1,\sigma_2\in\mathcal{T} \Rightarrow \sigma_1\vee\sigma_2\in\mathcal{T},$$
and
$$\sigma_1,\sigma_2\in\mathcal{T} \Rightarrow \sigma_1\wedge\sigma_2\in\mathcal{T},$$
for any stopping time $\sigma_2$. Let $\mathfrak{t}=\sup_{\sigma\in\mathcal T} \sigma$.
Then
we may choose an increasing sequence $(\sigma_M)\subset\mathcal{T}$ such that $\lim_{M\to\infty}\sigma_M=\mathfrak{t}$ a.s. Let $[\vr_M,\vu_M]$ be the corresponding sequence of solutions on $[0,\sigma_M]$. Due to uniqueness, this sequence defines a solution $(\vr,\vu)$ on $\cup_M[0,\sigma_M]$ by setting $(\vr,\vu):=(\vr_M,\vu_M)$ on $[0,\sigma_M]$.
For each $R\in\mn$ we now define
\[
\tau_R=\mathfrak{t} \wedge \inf \left\{ t \in [0,T]\ \Big| \ \| \vu (t) \|_{2,\infty} \geq R \right\}.
\]
Then $(\vr,\vu)$ is a solution on $[0,\sigma_M\wedge\tau_R]$ and sending $M\to\infty$ we obtain that $(\vr,\vu)$ is a solution on $[0,\tau_R]$. Note that $\tau_R$ is not a.s. strictly positive unless $\|\vu_0\|_{2,\infty}<R.$ Nevertheless, since $\vu_0\in W^{s,2}(\mt,\mr^N)$ a.s. we may deduce that for almost every $\omega$ there exists $R=R(\omega)$ such that $\mathfrak{t}_{R(\omega)}(\omega)>0$.
To guarantee the strict positivity, we combine the two sequences of stopping times $(\sigma_R)$ and $(\tau_R)$ and define $\mathfrak{t}_R=\sigma_R\vee\tau_R$. Then each triplet $(\vr,\vu,\mathfrak{t}_R)$, $R\in\mn$, is a local strong pathwise solution with an a.s. strictly positive stopping time. Next,  we observe that, by repeating the construction of a local strong pathwise solution,  a solution on $[0,\mathfrak{t}_R]$ can be extended to a solution on $[0,\mathfrak{t}_R+\sigma]$ for an a.s. strictly positive stopping time $\sigma$. Thus, in order to show that $\mathfrak{t}_R<\mathfrak{t}$ on $[\mathfrak{t}<T]$, assume for a contradiction that $\p(\mathfrak{t}_R=\mathfrak{t}<T)>0$. Then we have $\mathfrak{t}_R+\sigma\in\mathcal{T}$ and hence $\p(\mathfrak{t}<\mathfrak{t}_R+\sigma)>0$ which contradicts the maximality of $\mathfrak{t}$. Consequently, $(\mathfrak{t}_R)$ is an increasing sequence of stopping times converging to $\mathfrak{t}$.
 Moreover, on the set $[\mathfrak{t}<T]$ we have that
$$\sup_{t\in[0,\mathfrak{t}_R]}\|\vu(t)\|_{2,\infty}\geq R.$$
Thus, the existence part  of Theorem \ref{thm:main} is complete.

\def\cprime{$'$} \def\ocirc#1{\ifmmode\setbox0=\hbox{$#1$}\dimen0=\ht0
  \advance\dimen0 by1pt\rlap{\hbox to\wd0{\hss\raise\dimen0
  \hbox{\hskip.2em$\scriptscriptstyle\circ$}\hss}}#1\else {\accent"17 #1}\fi}

\end{document}